\documentclass[numbers=enddot,12pt,final,onecolumn,notitlepage]{scrartcl}%
\usepackage[headsepline,footsepline,manualmark]{scrlayer-scrpage}
\usepackage{amssymb}
\usepackage{amsmath}
\usepackage{amsthm}
\usepackage{framed}
\usepackage{comment}
\usepackage{ytableau}
\usepackage[maxbibnames=99,backend=bibtex]{biblatex}
\usepackage[breaklinks=True]{hyperref}
\usepackage[sc]{mathpazo}
\usepackage[T1]{fontenc}
\usepackage{needspace}
\usepackage{tabls}
\providecommand{\U}[1]{\protect\rule{.1in}{.1in}}
\theoremstyle{definition}
\newtheorem{theo}{Theorem}[section]
\newenvironment{theorem}[1][]
{\begin{theo}[#1]\begin{leftbar}}
{\end{leftbar}\end{theo}}
\newtheorem{lem}[theo]{Lemma}
\newenvironment{lemma}[1][]
{\begin{lem}[#1]\begin{leftbar}}
{\end{leftbar}\end{lem}}
\newtheorem{prop}[theo]{Proposition}
\newenvironment{proposition}[1][]
{\begin{prop}[#1]\begin{leftbar}}
{\end{leftbar}\end{prop}}
\newtheorem{defi}[theo]{Definition}
\newenvironment{definition}[1][]
{\begin{defi}[#1]\begin{leftbar}}
{\end{leftbar}\end{defi}}
\newtheorem{remk}[theo]{Remark}
\newenvironment{remark}[1][]
{\begin{remk}[#1]\begin{leftbar}}
{\end{leftbar}\end{remk}}
\newtheorem{coro}[theo]{Corollary}
\newenvironment{corollary}[1][]
{\begin{coro}[#1]\begin{leftbar}}
{\end{leftbar}\end{coro}}
\newtheorem{conv}[theo]{Convention}
\newenvironment{convention}[1][]
{\begin{conv}[#1]\begin{leftbar}}
{\end{leftbar}\end{conv}}
\newtheorem{conj}[theo]{Conjecture}
\newenvironment{conjecture}[1][]
{\begin{conj}[#1]\begin{leftbar}}
{\end{leftbar}\end{conj}}
\newtheorem{exam}[theo]{Example}
\newenvironment{example}[1][]
{\begin{exam}[#1]\begin{leftbar}}
{\end{leftbar}\end{exam}}
\newenvironment{verlong}{}{}
\newenvironment{vershort}{}{}
\excludecomment{verlong}
\includecomment{vershort}
\newcommand{\abs}[1]{\left| #1 \right|}
\newcommand{\tup}[1]{\left( #1 \right)}
\newcommand{\Fq}{\mathbb{F}_q}
\renewcommand{\leq}{\leqslant}
\renewcommand{\geq}{\geqslant}
\ihead{Multislant matrices and Jacobi--Trudi determinants over finite fields}
\ohead{page \thepage}
\begin{document}

\title{Multislant matrices and Jacobi--Trudi determinants over finite fields}
\author{Omesh Dhar Dwivedi, Jonah Blasiak, Darij Grinberg}
\date{June 5, 2023}
\maketitle

\textbf{Abstract.}
The problem of counting the $\mathbb{F}_q$-valued points of a variety has been well-studied from algebro-geometric, topological, and combinatorial perspectives.  We explore a combinatorially flavored version of this problem studied by Anzis et al. \cite{Anzis18}, which is similar to work of Kontsevich \cite{Kontsevich}, Elkies \cite{Elkies}, and Haglund \cite{Haglund}.

Anzis et al. considered the question: what is the probability that the determinant of a Jacobi-Trudi matrix vanishes if the variables are chosen uniformly at random from a finite field?
They gave a formula for various partitions such as hooks, staircases, and rectangles.
We give a formula for partitions whose parts form an arithmetic progression,
verifying and generalizing one of their conjectures.
More generally, we compute the probability of the determinant vanishing for a class of matrices (``multislant matrices'') made of Toeplitz blocks with certain properties.

We furthermore show that the determinant of a skew Jacobi-Trudi matrix is equidistributed across the finite field if the skew partition is a ribbon.

\bigskip\noindent\textbf{Keywords:} Jacobi--Trudi matrices, Schur functions, finite fields, determinants, Toeplitz matrices.

\bigskip\noindent\textbf{MSC classes (2020 Mathematics Subject Classification):} 05E05, 15B05, 11T06, 11C20.

\section{\label{sec.introduction}Introduction}

Let $\mathbb{F}_q$ be a finite field.

A \emph{Jacobi--Trudi matrix} is a matrix of the form
\begin{align*}
J_{\lambda_1, \lambda_2, \ldots, \lambda_k}\tup{z_1, z_2, z_3, \ldots}
& := 
\left(  z_{\lambda_{i}-i+j}\right)_{1\leq i\leq k,\ 1\leq j\leq k}
\\
&= \begin{pmatrix}
z_{\lambda_1} & z_{\lambda_1 + 1} & \cdots & z_{\lambda_1 + k-1} \\
z_{\lambda_2 - 1} & z_{\lambda_2} & \cdots & z_{\lambda_2 + k-2} \\
\vdots & \vdots & \ddots & \vdots \\
z_{\lambda_k - k+1} & z_{\lambda_k - k+2} & \cdots & z_{\lambda_k}
\end{pmatrix} ,
\end{align*}
where $\lambda_{1} \geq \lambda_{2} \geq \cdots \geq \lambda_{k} \geq 0$ are integers.
Here, $z_1, z_2, z_3, \ldots$ are elements of $\Fq$ chosen at random (uniformly and independently), and we furthermore set $z_{0}=1$ and $z_{i}=0$ for $i<0$.

We study the probability that the determinant
$\det\tup{ J_{\lambda_1, \lambda_2, \ldots, \lambda_k}\tup{z_1, z_2, z_3, \ldots} }$
of this matrix vanishes.
This probability was computed in
\cite{Anzis18} for certain families of $\lambda_i$'s.
In particular, this probability is $\dfrac{1}{q}$ in the following cases:
\begin{itemize}
\item the ``\emph{rectangle case}'': $\lambda_i = m$ for each $i$ (see \cite[Corollary 6.4]{Anzis18}, or \cite{dwivedi2021rank} for a generalization);
\item the ``\emph{staircase case}'': $\lambda_i = k+1-i$ for each $i$ (see \cite[Theorem 6.5]{Anzis18});
\item the ``\emph{hook case}'': $\lambda_1$ arbitrary; $\lambda_i = 1$ for
each $i = 2, 3, \ldots, k$ (see \cite[Proposition 1.2]{Anzis18}).
\end{itemize}
In more complicated situations, the probability can be less well-behaved,
although Anzis et al. prove \cite[Corollary 5.2]{Anzis18} that it is
always greater or equal to $\dfrac{1}{q}$ and conjecture
\cite[Conjecture 5.10]{Anzis18} that it is always less than or equal to
\[
1 - \abs{\operatorname{GL}_k\tup{\Fq}} / q^{k^2}
= 1 - \prod_{i=1}^k \tup{1 - \dfrac{1}{q^i}} .
\]

The main result of the present paper is a formula for this probability
in yet another case (proved in \S\ \ref{sec.mulslant}):

\begin{theorem}
\label{thm.n.staircase-intro}
Let
\[\left(\lambda_1, \lambda_2, \ldots, \lambda_k\right)
= (p+(k-1)n, \ p+(k-2)n, \ \ldots, \ p+n, \  p) , \]
for some integers $p$, $n$ and $k$ satisfying $0 < p \leq n \leq k-1$.
Then, the probability of
$\det\tup{ J_{\lambda_1, \lambda_2, \ldots, \lambda_k}\tup{z_1, z_2, z_3, \ldots} }$
vanishing is
\[
1- \prod_{i=1}^{n} \left(  1-\dfrac{1}{q^{i}}\right) .
\]
\end{theorem}

Mysteriously, the right hand side here is also the probability of a completely random generic $n\times n$-determinant vanishing.


Underlying
our proof is a much more general statement that applies to a broader
class of determinants with similar structures (determinants of
what we call ``multislant matrices'').

Theorem~\ref{thm.n.staircase-intro} proves and generalizes \cite[Conjecture 10.1]{Anzis18} (and also confirms the upper bound conjecture \cite[Conjecture 5.10]{Anzis18} for the class of partitions covered by this theorem).

Theorem~\ref{thm.n.staircase-intro} will be proved as part \textbf{(ii)} of Theorem~\ref{thm.n.staircase} further below. Part \textbf{(i)} of the same theorem will provide a similar answer in the case when $k < n+1$. The case $p > n$ appears to give rise to a similar formula for the vanishing probability (Conjecture~\ref{bigoutwardstaircases}), but one we have been unable to prove.
\\

A crucial motivation for studying the determinants of these Jacobi--Trudi matrices
is the theory of \emph{Schur functions}:
If we replace the $z_{i}$ by the generators $h_{i}$ of the
ring of symmetric functions (see \S~\ref{sec.def} for a definition), then
$\det\tup{J_{\lambda_1, \lambda_2, \ldots, \lambda_k}\tup{z_1, z_2, z_3, \ldots}}$
becomes a Schur function.
This is the viewpoint taken by \cite{Anzis18} when studying these determinants.

Our results can also be viewed from a geometric perspective:
We are counting $\Fq$-valued points on certain affine schemes;
similar questions have been studied by Weil \cite{Weil}, Kontsevich \cite{Kontsevich}, Stembridge \cite{Stembridge}, Stanley \cite{RPStan}, Belkale-Brosnan \cite{Belkale2003}, Elkies \cite{Elkies}, Haglund \cite{Haglund}, and many others.
A question often asked is whether the number of $\Fq$-valued points has the form $f\tup{q}$ for some Laurent polynomial $f$; this is not always satisfied for $\det\tup{J_{\lambda_1, \lambda_2, \ldots, \lambda_k}\tup{z_1, z_2, z_3, \ldots}}$ (see \cite[Proposition 5.7]{Anzis18}), but is true in surprisingly many situations, including the ones we analyze here. (The same question has been asked by Kontsevich for certain graph-theoretical polynomials, and answered in the negative \cite{Belkale2003}.)
\\

We note that the values of $\det\tup{ J_{\lambda_1, \lambda_2, \ldots, \lambda_k}\tup{z_1, z_2, z_3, \ldots} }$ are not equidistributed in $\Fq$ in general, even when the value $0$ is taken with probability $\dfrac{1}{q}$. For example, equidistribution is satisfied in the hook case (\cite[Proposition 1.2]{Anzis18}) but not in the rectangle case (\cite[Lemma 9.3]{Anzis18}). Generalizing the hook case, we show that equidistribution holds for \emph{ribbons} (a class of skew partitions).
We also prove a certain symmetry property (Theorem~\ref{thm.transpose}) for the probabilities of our determinants taking certain values.

\subsubsection*{Acknowledgments}

We thank the referees for helpful comments that improved the clarity of the present paper.

\section{\label{sec.def}Notations and background}

Fix a finite field $\Fq$. Consider the polynomial ring
\[
\mathcal{P}  :=  \mathbb{Z}\left[  h_{1},h_{2},h_{3},\ldots\right]
\]
in countably many indeterminates $h_{1},h_{2},h_{3},\ldots$.

\begin{definition}
\label{def.mapstozero}
For any $f\in\mathcal{P}$ and $a\in \Fq$,
let $N\in\mathbb{N}$ be such that $f$ only
involves the indeterminates $h_{1},h_{2},\ldots,h_{N}$.
Define the \emph{probability of $f$ evaluating to $a$}
to be the rational number
\[
P\left(  f\mapsto a\right)   := \dfrac{\left(  \text{\# of }\left(  z_{1}%
,z_{2},\ldots,z_{N}\right)  \in \Fq^{N} \mid f\left(
z_{1},z_{2},\ldots,z_{N}\right)  =a\right)  }{\left(  \text{\# of all }\left(
z_{1},z_{2},\ldots,z_{N}\right)  \in \Fq^{N}\right)  }.
\]
\end{definition}

Note that the right hand side is independent on the choice of $N$
because increasing $N$ by $1$
merely multiplies the numerator and the denominator by $q$.

As the name suggests, $P\tup{f \mapsto a}$ is the probability that
$P$ evaluates to $a$ when $h_1, h_2, h_3, \ldots$ are specialized
to randomly chosen elements of $\Fq$ (chosen uniformly and
independently). Since $f$ involves only finitely many indeterminates,
this does not actually require choosing infinitely many random
elements.

We can identify the polynomial ring
$\mathcal{P}=\mathbb{Z}\left[  h_{1},h_{2},h_{3},\ldots\right]  $
with the ring $\Lambda$ of ``\emph{symmetric
functions}'', which are in fact symmetric formal power series in
countably many indeterminates $x_{1},x_{2},x_{3},\ldots$
having bounded degree (see \cite[Section 7.1]{EC2} or
\cite[Section I.2]{Macdonald}).
To do so, we equate each generator $h_i$ of $\mathcal{P}$
with the complete homogeneous symmetric function
$h_i\tup{x_1, x_2, x_3, \ldots}
= \sum\limits_{j_1 \leq j_2 \leq \ldots \leq j_i}
x_{j_1} x_{j_2} \cdots x_{j_i} \in \Lambda$.
In this paper, we will not actually use the $x_i$.
However, we will use some of the structure of $\Lambda$, in
particular the omega involution $\omega \colon \Lambda \to \Lambda$
(see, e.g., \cite[Section 7.6]{EC2} or \cite[(I.2.7)]{Macdonald}).

\begin{definition}
A \emph{partition} $\lambda = \left(\lambda_1, \lambda_2, \ldots, \lambda_k\right)$ is a
weakly decreasing finite sequence of positive integers, with
$\lambda_1 \geq \lambda_2 \geq \cdots \geq \lambda_k$.
\end{definition}


There are several important concepts associated with a partition.

\begin{definition}
The \emph{size} of a partition
$\lambda = \left(\lambda_1, \lambda_2, \ldots, \lambda_k\right)$
is the sum $\lambda_1 + \lambda_2 + \cdots + \lambda_k$.
It is denoted by $\abs{\lambda}$.

The \emph{length} of a partition
$\lambda = \left(\lambda_1, \lambda_2, \ldots, \lambda_k\right)$
is the number $k$.
It is denoted by $\ell\tup{\lambda}$.
\end{definition}

\begin{definition}
The \emph{Young diagram} $Y\tup{\lambda}$ of a partition
$\lambda = \left(\lambda_1, \lambda_2, \ldots, \lambda_k\right)$
is a table (not necessarily of rectangular shape).
Its boxes (also known as cells) are arranged in
$k$ left-justified rows, with the $i$-th row (counted from the top)
having $\lambda_i$ boxes. (Thus, the total number of boxes
is $\abs{\lambda}$.)

Formally, $Y\tup{\lambda}$ is defined as the set of all pairs
$\tup{i, j}$ of positive integers satisfying $i \leq k$ and
$j \leq \lambda_i$. The pair $\tup{i, j}$ corresponds to the
$j$-th box (from the left) in the $i$-th row of the diagram.
\end{definition}

\begin{definition}
The \emph{conjugate} $\lambda^t$ of a partition $\lambda$ is defined to be the partition whose Young diagram is obtained from $Y\tup{\lambda}$ by flipping it across the main diagonal (i.e., the length of the $i$-th row of $Y\tup{\lambda^t}$ equals the length of the $i$-th column of $Y\tup{\lambda}$ for each $i$).
\end{definition}

Example \ref{ydiagexample} below shows the construction of a Young diagram for the partition $\lambda = (7,4,1)$ and its conjugate. Note that the number of boxes in any row is the corresponding part of $\lambda$. Since a partition is a weakly decreasing sequence, the number of boxes in a row decreases as we go down the diagram.

\Needspace{10pc}

\begin{example}
\label{ydiagexample}
The conjugate of $\lambda = (7,4,1)$ is $\lambda^t =(3,2,2,2,1,1,1)$.
Here are the Young diagrams of these two partitions:
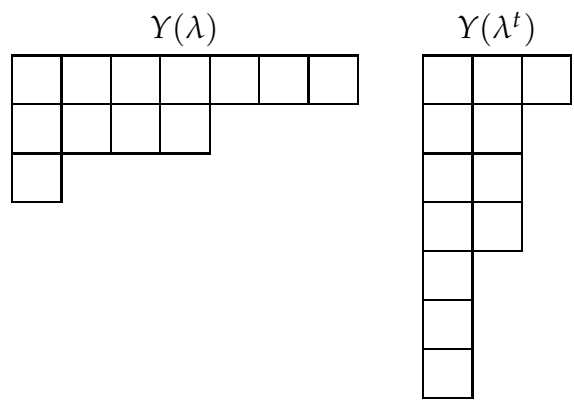
\[
\begin{tabular}{ccc}
$Y(\lambda)$ & & $Y(\lambda^t)$ \\
\ydiagram{7,4,1}
&
&
\ydiagram{3,2,2,2,1,1,1}
\end{tabular}
\]
\end{example}



\begin{definition}
If $\lambda = \tup{\lambda_1, \lambda_2, \ldots, \lambda_k}$ is
a partition, then we set $\lambda_i  :=  0$ for each $i > k$.
Thus, the partition $\lambda$ is identified with the infinite
weakly decreasing
sequence $\tup{\lambda_1, \lambda_2, \lambda_3, \ldots}$
whose entries stabilize at $0$.
\end{definition}

A generalization of partitions are \emph{skew partitions}.

\begin{definition}

\begin{enumerate}

\item[\textbf{(a)}]
A \emph{skew partition} is a pair of partitions $(\lambda, \mu)$ such that $Y\tup{\mu} \subseteq Y\tup{\lambda}$ (or, equivalently, such that $\mu_i \leq \lambda_i$ for each $i \geq 1$); it is denoted by $\lambda / \mu$.
We shall also use the shorthand $\mu \subseteq \lambda$ for the condition $Y\tup{\mu} \subseteq Y\tup{\lambda}$ here.

\item[\textbf{(b)}]
The \emph{skew diagram} $Y\tup{\lambda / \mu}$ of a skew partition $\lambda / \mu$ is defined as the set-theoretic difference $Y\tup{\lambda} \setminus Y\tup{\mu}$ of the Young diagrams of $\lambda$ and $\mu$: the set of boxes that belong to the diagram of $\lambda$ but not to that of $\mu$.

\item[\textbf{(c)}]
Any partition $\lambda$ is identified with the skew partition $\lambda / \varnothing$, where $\varnothing  :=  \tup{0,0,0,\ldots}$ is the \emph{empty partition}.

\end{enumerate}

\end{definition}

We shall now define the Jacobi--Trudi matrix of a skew partition (generalizing \cite[Definition 2.3]{Anzis18}):

\begin{definition}[Jacobi--Trudi matrix; skew Schur function]
\label{Jacobi--Trudi Identity}

\begin{enumerate}

\item[\textbf{(a)}]
Let $\lambda / \mu$ be a skew partition, and let $k = \ell\tup{\lambda}$.
Define the \emph{Jacobi--Trudi matrix} of $\lambda/\mu$ to be the $k \times k$-matrix
\[
J(\lambda/\mu)  :=  \tup{ h_{\lambda_i-\mu_j - i + j} }_{1\leq i\leq k,\ 1\leq j\leq k}
\in \mathcal{P}^{k\times k} .
\]
Here, we set $h_0  :=  1$ and $h_m  :=  0$ for all $m < 0$.

\item[\textbf{(b)}]
The \emph{skew Schur function} of $\lambda/\mu$ is defined to be
\begin{align}
   s_{\lambda/\mu}  :=  \det \tup{ J(\lambda/\mu) } \  \in \mathcal{P} .
\label{eq.def-slm}
\end{align}

\item[\textbf{(c)}]
When $\mu = \varnothing$, we denote the matrix $J(\lambda/\mu)$ and the skew Schur function $s_{\lambda/\mu}$ as $J(\lambda)$ and $s_\lambda$, respectively.

\end{enumerate}

\end{definition}

\begin{remark}
\label{flushremark}
If $\mathcal{P}$ is viewed as the ring of the symmetric functions, then the skew Schur function $s_{\lambda/\mu}$ is often defined combinatorially as follows: 
A \emph{semistandard tableau} of shape $\lambda / \mu$ is a filling of the Young diagram of $\lambda/\mu$ by positive integers which weakly increase along rows and strictly increase down columns.
Then $s_{\lambda/\mu}$ is the sum of monomials
\[
\mathbf{x}_T = \prod_{i=1}^{\infty} x_i^{\text{\# of $i$'s in }T}
\]
over all semistandard tableaux $T$ of shape $\lambda / \mu$.
With this definition, \eqref{eq.def-slm} is a theorem, known as the \emph{first Jacobi--Trudi identity} (see, e.g., \cite[Theorem 7.16.1]{EC2}).
We have chosen to take \eqref{eq.def-slm} as the definition of $s_{\lambda/\mu}$ since we will 
only need the combinatorial description of $s_{\lambda/\mu}$ in a minor way in later proofs.
%
\end{remark}

For example, if $\lambda$ is the partition $(7,4,1)$ as in Example \ref{ydiagexample}, then the Schur function $s_\lambda$ is given by
\begin{align*}
s_{(7,4,1)} =
\det
\begin{pmatrix}
h_7 & h_8 & h_9 \\
h_3 & h_4 & h_5 \\
0 & 1 & h_1 \\
\end{pmatrix} ,
\end{align*}
and the matrix on the right hand side of this equality is $J\tup{\lambda}$.


\begin{remark}
Let
$\lambda = \tup{\lambda_1, \lambda_2, \ldots, \lambda_k}$
be a partition.
Then, the matrix
$J_{\lambda_1, \lambda_2, \ldots, \lambda_k}\tup{z_1, z_2, z_3, \ldots}$
from \S~\ref{sec.introduction} is precisely the matrix
$J\tup{\lambda}$, after each $h_i$ has been specialized to $z_i$.
\end{remark}

\section{\label{sec.mulslant}Multislant matrices and $p$-shifted $n$-staircases}

\subsection{$p$-shifted $n$-staircases}

The first type of partitions that we will study are the
$p$-shifted $n$-staircases.
These are precisely the partitions whose entries form a non-constant
arithmetic progression.
Explicitly, they are defined as follows:

\begin{definition}
Let $p > 0$, $n > 0$ and $k \geq 0$ be integers.
The \emph{$p$-shifted $n$-staircase of length $k$}
is the partition
\[
\lambda = (p+(k-1)n, \ p+(k-2)n, \ \ldots, \ p+n, \  p) .
\]
When $p \leq n$, we call this partition an \emph{inward-shifted $n$-staircase};
when $p > n$, we call it an \emph{outward-shifted $n$-staircase}.

\end{definition}

The name, of course, refers to the shape of the Young diagram.


Anzis et al. \cite{Anzis18} conjectured that a $2$-shifted $2$-staircase $\lambda$ satisfies $P(s_{\lambda} \mapsto 0) = \dfrac{q^2+q -1}{q^3}$. We prove this result and generalize it to arbitrary $n$ and $p$:

\begin{theorem}
\label{thm.n.staircase}

Let $\lambda$ be the $p$-shifted $n$-staircase of
length $k$ (for given $p > 0$, $n > 0$ and $k \geq 0$).

\begin{enumerate}

\item[\textbf{(i)}]
Assume that $k < n+1$ (with $p$ arbitrary). Then,
\[
P(s_{\lambda} \mapsto 0) =
1 -
\begin{cases}
\prod\limits_{i=1}^{k-1} \left(  1-\dfrac{1}{q^{i}}\right) ,
& \text{ if } p \leq k-1; \\
\prod\limits_{i=1}^{k} \left(  1-\dfrac{1}{q^{i}}\right) ,
& \text{ if } p > k-1.
\end{cases}
\]

\item[\textbf{(ii)}]
Assume that $p \leq n$ and $k \geq n+1$. Then,
\[
P(s_{\lambda} \mapsto 0) = 1- \prod_{i=1}^{n} \left(  1-\dfrac{1}{q^{i}}\right) .
\]

\end{enumerate}

\end{theorem}

\begin{conjecture}
\label{bigoutwardstaircases}
Let $\lambda$ be the $p$-shifted $n$-staircase of length
$k \geq n+1$ with $p > n$.  Then,
\[
P(s_{\lambda} \mapsto 0) = 1-\prod_{i=1}^{n+1} \left(  1-\dfrac{1}{q^{i}}\right).
\]
\end{conjecture}

Theorem \ref{thm.n.staircase} and Conjecture \ref{bigoutwardstaircases} together encompass the whole class of $p$-shifted $n$-staircases. We note that none of the four expressions for $P(s_{\lambda} \mapsto 0)$ depend on $p$.
Theorem \ref{thm.n.staircase} \textbf{(ii)} restates Theorem \ref{thm.n.staircase-intro} from the introduction.

The right hand side of Theorem~\ref{thm.n.staircase} \textbf{(ii)}
is the probability of a random $n\times n$-matrix being singular over $\mathbb{F}_q$.

We shall prove Theorem~\ref{thm.n.staircase} \textbf{(ii)} in Subsection \ref{subsec.proof-shifted-staircase}. This proof will rely on a more general result about a certain kind of structured matrices (consisting of rectangular Toeplitz blocks glued together along their vertical sides), which we call \emph{multislant matrices}, and which are defined in Definition~\ref{def.multislant} below. These matrices exhibit a well-behaved recursive structure, which allows us to explicitly compute the probability for their determinants to vanish (Theorem~\ref{thm.multislant.prob}). Having done that, we will be able to prove Theorem~\ref{thm.n.staircase} \textbf{(ii)} by identifying the appropriate Jacobi--Trudi matrix $J\tup{\lambda}$ as a multislant matrix after permuting its columns (Lemma~\ref{multistair}).

Theorem~\ref{thm.n.staircase} \textbf{(i)} can be proved similarly,
but is also fairly easy to check directly, since all indeterminates
in the matrix $J\tup{\lambda}$ are distinct (and all entries of this
matrix are indeterminates, except for some possible $0$'s and $1$'s
in the last row).

We have so far been unable to adapt this type of reasoning to Conjecture~\ref{bigoutwardstaircases}; it appears to require a looser notion of ``multislant matrices'', which no longer lends itself to an easy computation of the vanishing probability.

\subsection{\label{subsec.multislant.multislant}Multislant matrices}

We let $\mathbb{N}$ denote the set $\left\{  0,1,2,\ldots\right\}  $. We let
$\left[  n\right]  =\left\{  1,2,\ldots,n\right\}  $ for each $n\in\mathbb{N}$.

All matrices appearing in the following are over $\Fq$ or over
polynomial rings over $\Fq$.



\begin{definition}
Let $A_{1},A_{2},\ldots,A_{k}$ be finitely many matrices with the same number
of rows. Then, $\left(  A_{1}\mid A_{2}\mid\cdots\mid A_{k}\right)  $ shall
denote the block matrix whose blocks are $A_{1},A_{2},\ldots,A_{k}$, arranged
horizontally from left to right.
\end{definition}

For example, if $A_{1}=\left(
\begin{array}
[c]{cc}%
a & b\\
c & d
\end{array}
\right)  $ and $A_{2}=\left(
\begin{array}
[c]{c}%
e\\
f
\end{array}
\right)  $ and $A_{3}=\left(
\begin{array}
[c]{cc}%
g & h\\
i & j
\end{array}
\right)  $, then $\left(  A_{1}\mid A_{2}\mid A_{3}\right)  =\left(
\begin{array}
[c]{ccccc}%
a & b & e & g & h\\
c & d & f & i & j
\end{array}
\right)  $.

In the following, an empty box in a matrix is always understood to be filled
with zero. For example, $\left(
\begin{array}
[c]{ccc}%
1 &  & \\
2 & 3 & \\
& 4 & 5
\end{array}
\right)  $ means the $3\times3$-matrix $\left(
\begin{array}
[c]{ccc}%
1 & 0 & 0\\
2 & 3 & 0\\
0 & 4 & 5
\end{array}
\right)  $.

\begin{definition}
A $u\times v$-matrix is said to be \emph{tall} if $u\geq v$.
\end{definition}

We shall now define a notion of ``Toeplitz matrices'' tailored to
our needs. (In the literature, Toeplitz matrices are usually square
and sometimes infinite, whereas ours will be finite but usually
non-square. Apart from these differences, our notion agrees with
the standard one.)

\begin{definition}
Let $A=\left(  a_{i,j}\right)  _{1\leq i\leq u,\ 1\leq j\leq v}$ be a $u\times
v$-matrix.

\begin{enumerate}

\item[\textbf{(a)}]

For each $k\in\mathbb{Z}$, the \emph{$k$-th paradiagonal} of $A$
will mean the list of all entries $a_{i,j}$ of $A$ with $i-j=k$. (These
entries are listed in the order of increasing $i$, or, equivalently, in the
order of increasing $j$.)

For example, the nonempty paradiagonals of the matrix $\left(
\begin{array}
[c]{cc}%
a & b\\
c & d\\
e & f
\end{array}
\right)  $ are%
\[
\underbrace{\left(  b\right)  }_{\left(  -1\right)  \text{-st paradiagonal}%
},\ \underbrace{\left(  a,d\right)  }_{0\text{-th paradiagonal}}%
,\ \underbrace{\left(  c,f\right)  }_{1\text{-st paradiagonal}}%
,\ \underbrace{\left(  e\right)  }_{2\text{-nd paradiagonal}}.
\]

\item[\textbf{(b)}]
A paradiagonal of $A$ is said to be \emph{full} if it has $v$ entries.

For example, the full paradiagonals of the matrix $\left(
\begin{array}
[c]{cc}%
a & b\\
c & d\\
e & f
\end{array}
\right)  $ are $\left(  a,d\right)  $ and $\left(  c,f\right)  $. Note that a
matrix that is not tall does not have any full paradiagonals.

\item[\textbf{(c)}]
If $u\geq v$, then the $\left(  u-v\right)  $-th paradiagonal of
a $u\times v$-matrix $A$ will also be called the \emph{bottommost full
paradiagonal} of $A$. (It is indeed full and is indeed the bottommost of the
full paradiagonals of $A$.)

\item[\textbf{(d)}]
The matrix $A$ is said to be \emph{Toeplitz} if each of its
paradiagonals consists of equal entries (i.e., if any two entries that belong
to the same paradiagonal are equal).

For example, a $4\times3$-matrix is Toeplitz if and only if it has the form
$\left(
\begin{array}
[c]{ccc}%
c & b & a\\
d & c & b\\
e & d & c\\
f & e & d
\end{array}
\right)  $ for some $a,b,c,d,e,f$.

\item[\textbf{(e)}]
The entries $a_{i,j}$ of $A$ with $i<j$ will be called the
\emph{attic entries} of $A$. For example, the only attic entry of the matrix
$\left(
\begin{array}
[c]{cc}%
a & b\\
c & d\\
e & f
\end{array}
\right)  $ is $b$.

\item[\textbf{(f)}]
The entries $a_{i,j}$ of $A$ will $i>j+u-v$ will be called the
\emph{basement entries} of $A$. For example, the only basement entry of the
matrix $\left(
\begin{array}
[c]{cc}%
a & b\\
c & d\\
e & f
\end{array}
\right)  $ is $e$.

\end{enumerate}

\end{definition}

In a tall matrix, the attic entries are the entries above the topmost
full paradiagonal, whereas the basement entries are the entries below
the bottommost full paradiagonal.

\begin{definition}
\label{def.slant} A \emph{slant matrix} means a tall Toeplitz matrix of the
form%
\[
\left(
\begin{array}
[c]{cccc}%
u_{m} & \ast & \ast & \ast\\
\vdots & \ddots & \vdots & \vdots\\
u_{2} & \ddots & u_{m} & \ast\\
u_{1} & \ddots & \vdots & u_{m}\\
u_{0} & \ddots & u_{2} & \vdots\\
& \ddots & u_{1} & u_{2}\\
&  & u_{0} & u_{1}\\
&  &  & u_{0}%
\end{array}
\right)  ,
\]
where

\begin{itemize}
\item each asterisk (``$\ast$'') is an indeterminate or an element of $\Fq$;

\item $u_{1},u_{2},\ldots,u_{m}$ are $m$ distinct indeterminates; and

\item $u_{0}$ is either an indeterminate distinct from $u_{1},u_{2}%
,\ldots,u_{m}$ or an element of $\Fq$;

\item none of the indeterminates among the attic entries appears on any
full paradiagonal.
\end{itemize}

We require $m\geq0$, and we require that the matrix have at least one column.

We say that the slant matrix shown above has \emph{type X} if $u_{0}$ is an
indeterminate; we say that it has \emph{type 0} if $u_{0}=0$; we say that it
has \emph{type 1} if $u_{0}\in \Fq\setminus\left\{  0\right\}  $.
\end{definition}

Thus, in a slant matrix,

\begin{itemize}
\item all attic entries are indeterminates or elements of $\Fq$,
and are equal along each paradiagonal (but can differ between different paradiagonals);

\item all basement entries are $0$;

\item each paradiagonal is constant (i.e., any two entries lying on the same
paradiagonal are equal);

\item the (equal) entries on each full paradiagonal are indeterminates, except
possibly on the bottommost full paradiagonal, whose entry can also be an
element of $\Fq$;

\item indeterminates on any two distinct paradiagonals are distinct, except
possibly if both of them are in the attic (i.e., on non-full paradiagonals).
\end{itemize}

For example, the matrices%
\[
\left(
\begin{array}
[c]{ccc}%
z & 2 & w\\
y & z & 2\\
x & y & z\\
& x & y\\
&  & x
\end{array}
\right)  ,\ \ \ \left(
\begin{array}
[c]{cc}%
z & 4\\
y & z\\
0 & y\\
& 0
\end{array}
\right)  ,\ \ \ \left(
\begin{array}
[c]{cccc}%
z & 3 & 1 & w\\
y & z & 3 & 1\\
4 & y & z & 3\\
& 4 & y & z\\
&  & 4 & y\\
&  &  & 4
\end{array}
\right)
\]
(where $x,y,z,w$ are four distinct indeterminates) are slant matrices of type
X, type 0 and type 1, respectively (assuming that $4\neq0$ in $\Fq$).

\begin{definition}
Two slant matrices are said to be \emph{disjoint} if there is no indeterminate
that appears in both of them.
\end{definition}

For example, the slant matrices $\left(
\begin{array}
[c]{cc}%
z & 3\\
y & z\\
0 & y\\
& 0
\end{array}
\right)  $ and $\left(
\begin{array}
[c]{cc}%
w & 3\\
x & w\\
0 & x\\
& 0
\end{array}
\right)  $ are disjoint (where $x,y,z,w$ are four distinct indeterminates),
whereas the slant matrices $\left(
\begin{array}
[c]{cc}%
z & 0\\
y & z\\
0 & y\\
& 0
\end{array}
\right)  $ and $\left(
\begin{array}
[c]{cc}%
z & 2\\
x & z\\
0 & x\\
& 0
\end{array}
\right)  $ are not.

\begin{definition}
\label{def.multislant} A \emph{multislant matrix} means a matrix $M$ of the
form $\left(  A_{1}\mid A_{2}\mid\cdots\mid A_{k}\right)  $, where
$A_{1},A_{2},\ldots,A_{k}$ are pairwise disjoint slant matrices with the same
number of rows. In this case, the slant matrices $A_{1},A_{2},\ldots,A_{k}$
are called the \emph{blocks} of $M$. Moreover, the \emph{signature} of $M$ is
defined to be the triple $\left(  i,j,\ell\right)  $, where

\begin{itemize}
\item $i$ is the number of blocks of $M$ that have type X,

\item $j$ is the number of blocks of $M$ that have type 0, and

\item $\ell$ is the number of blocks of $M$ that have type 1.
\end{itemize}

\noindent(Of course, $i+j+\ell=k$ in this case.)
\end{definition}

For example, the matrix%
\[
\left(
\begin{array}
[c]{ccccccccccccc}%
z & 1 & 0 & p & w & 3 & 3 & e & 0 & t & 3 & 4 & 9\\
y & z & 1 & 0 & v & w & 3 & d & e & s & t & 3 & 4\\
0 & y & z & 1 & u & v & w & c & d & 3 & s & t & 3\\
& 0 & y & z & 1 & u & v & b & c &  & 3 & s & t\\
&  & 0 & y &  & 1 & u & a & b &  &  & 3 & s\\
&  &  & 0 &  &  & 1 &  & a &  &  &  & 3
\end{array}
\right)
\]
(where all letters are distinct indeterminates) is a multislant matrix with
four blocks. The signature of this multislant matrix is $\left(  1,1,2\right)
$ if $3\neq0$ in $\Fq$ (since it has $1$ block of type X, $1$ block
of type 0 and $2$ blocks of type 1), and is $\left(  1,2,1\right)  $ if $3=0$
in $\Fq$.

Note that the (empty) $0\times0$-matrix is a multislant matrix (with $0$
blocks and signature $\left(  0,0,0\right)  $), but not a slant matrix. We
recall that the determinant of this $0\times0$-matrix is $1$ (by definition).

\begin{definition}
\label{def.SiPr}Let $M$ be a multislant matrix that is square. The
\emph{singular probability} of $M$ is defined to be the probability that $\det
M$ becomes $0$ if we substitute a random element of $\Fq$ for each
of the indeterminates appearing in $M$. (Here, the random elements of
$\Fq$ are meant to be chosen uniformly and independently.) The
singular probability of $M$ will be denoted by $\operatorname*{SiPr}M$.
\end{definition}

For example, if $M$ is the multislant matrix $\left(
\begin{array}
[c]{ccc}%
b & 3 & y\\
a & b & x\\
& a & 1
\end{array}
\right)  $ (for four distinct indeterminates $a,b,x,y$), then the singular
probability $\operatorname*{SiPr}M$ of $M$ is the probability that four
(uniformly and independently) random elements $\alpha,\beta,\xi,\theta$ of
$\Fq$ satisfy $\det\left(
\begin{array}
[c]{ccc}%
\beta & 3 & \theta\\
\alpha & \beta & \xi\\
& \alpha & 1
\end{array}
\right)  =0$. It turns out that this probability is precisely $\dfrac{1}{q}$.
More generally, we shall show a formula for the singular probability of any
multislant matrix.

First, we need a notation:

\begin{definition}
\label{def.gammak}For each positive integer $k$, we set%
\[
\gamma_{k} := \left(  1-\dfrac{1}{q^{k-1}}\right)  \left(  1-\dfrac{1}{q^{k-2}%
}\right)  \cdots\left(  1-\dfrac{1}{q^{1}}\right)  .
\]
Thus, in particular, $\gamma_{1}=1$ and $\gamma_{2}=1-\dfrac{1}{q}$. We also
set $\gamma_{0}=1$.
\end{definition}

It is well-known that $\gamma_{k}=\left\vert \operatorname*{GL}\nolimits_{k-1}%
\left(  \Fq\right)  \right\vert /\left\vert \operatorname*{M}%
\nolimits_{k-1}\left(  \Fq\right)  \right\vert $ for every positive
integer $k$. (But this will actually be a particular case of Theorem
\ref{thm.multislant.prob} below.)

Note that using the notations of $q$-calculus (specifically, the
$q$-Pochhammer symbol $\tup{a;\ q}_n$ as defined, e.g., in
\cite[\S 1.2]{Johnson2020}), we can rewrite the definition of $\gamma_k$ as
$\gamma_k = \tup{1/q;\ 1/q}_{k-1}$.

Now, we claim the following:

\begin{theorem}
\label{thm.multislant.prob}Let $M$ be a multislant matrix that is square and
nontrivial. Let $\left(  i,j,\ell\right)  $ be the signature of $M$, and let
$k=i+j+\ell$ be the number of blocks of $M$. Then,%
\begin{equation}
\operatorname*{SiPr}M=1-\gamma_{k}\left(  1-\dfrac{0^{\ell}}{q^{i}}\right)  .
\label{eq.thm.multislant.prob.eq}%
\end{equation}

\end{theorem}

Of course,
\[
0^{\ell}=%
\begin{cases}
0, & \text{if }\ell>0;\\
1, & \text{if }\ell=0.
\end{cases}
\]
Thus, (\ref{eq.thm.multislant.prob.eq}) can be restated as follows:

\begin{itemize}
\item If $\ell>0$ (that is, if $M$ has at least one block of type 1), then
$\operatorname*{SiPr}M=1-\gamma_{k}$.

\item If $\ell=0$ (that is, if $M$ has no block of type 1), then
$\operatorname*{SiPr}M=1-\gamma_{k}\left(  1-\dfrac{1}{q^{i}}\right)  $.
\end{itemize}

Theorem \ref{thm.multislant.prob} is proved by induction, based on a few
lemmas. First, we need some more notations:

\begin{definition}
Let $A$ be a slant matrix. Let $u_{0}$ be the bottommost entry in the last
column of $A$. Thus, $u_{0}$ is either an indeterminate (if $A$ has type X) or
$0$ (if $A$ has type 0) or a nonzero element of $\Fq$ (if $A$ has
type 1). Moreover, the bottommost full paradiagonal of $A$ is $\left(
u_{0},u_{0},\ldots,u_{0}\right)  $.

\begin{enumerate}

\item[\textbf{(a)}]
We call $u_{0}$ the \emph{bottom element} of $A$.

\item[\textbf{(b)}]
For any $v\in\Fq$, we let $A^{\rightarrow v}$ denote
the result of replacing all entries on the bottommost full paradiagonal of $A$
by $v$. Thus, $A^{\rightarrow v}$ is a slant matrix of type 0 if $v=0$, and
otherwise is a slant matrix of type 1.

\end{enumerate}

\end{definition}

For example, if $A=\left(
\begin{array}
[c]{ccc}%
z & 3 & 2\\
y & z & 3\\
x & y & z\\
& x & y\\
&  & x
\end{array}
\right)  $, then the bottom element of $A$ is the indeterminate $x$, and we
have%
\[
A^{\rightarrow0}=\left(
\begin{array}
[c]{ccc}%
z & 3 & 2\\
y & z & 3\\
0 & y & z\\
& 0 & y\\
&  & 0
\end{array}
\right)  \ \ \ \ \ \ \ \ \ \ \text{and}\ \ \ \ \ \ \ \ \ \ A^{\rightarrow
5}=\left(
\begin{array}
[c]{ccc}%
z & 3 & 2\\
y & z & 3\\
5 & y & z\\
& 5 & y\\
&  & 5
\end{array}
\right)  .
\]

\begin{definition}

\begin{enumerate}

\item[\textbf{(a)}]
A slant matrix is said to be \emph{strict} if all its attic
entries are elements of $\Fq$ (rather than indeterminates).

\item[\textbf{(b)}]
A multislant matrix is said to be \emph{strict} if all its blocks
are strict.

\end{enumerate}

\end{definition}

For example, the multislant matrix $\left(
\begin{array}
[c]{cccccc}%
z & t & 2 & v & 1 & s\\
y & z & t & u & v & r\\
x & y & z & w & u & q\\
& x & y & 1 & w & p\\
&  & x &  & 1 & 0
\end{array}
\right)  $ is not strict, since its first block has the indeterminate $t$ in
its attic; replacing this indeterminate $t$ by any constant element of
$\Fq$ would make this matrix strict.

Now, we can state the lemmas that will help us prove Theorem
\ref{thm.multislant.prob} by induction:

\begin{lemma}
\label{lem.multislant.red11}Let $M$ be a strict multislant matrix that is
square. Assume that $M$ has at least two blocks of type 1. Let $A_{i}$ and
$A_{j}$ be two blocks of $M$ that have type 1, with $i\neq j$. Assume that the
block $A_{i}$ has at least as many columns as the block $A_{j}$. Then, there
is a strict multislant matrix $M^{\prime}$ with the following properties:

\begin{itemize}
\item The matrix $M^{\prime}$ is square and has the same size as $M$.

\item It satisfies $\operatorname*{SiPr}M=\operatorname*{SiPr}\left(
M^{\prime}\right)  $.

\item The matrix $M^{\prime}$ differs from $M$ only in the block $A_{j}$ being
replaced by a new block, which has type 0.
\end{itemize}
\end{lemma}

\begin{proof}
First, we sketch the argument; then, we will illustrate it on an example.

The block $A_{i}$ has at least as many columns as the block $A_{j}$, and thus
has at most as many full paradiagonals as the block $A_{j}$ (since the two
blocks have the same height). Thus, the block $A_{i}$ has at most as many
indeterminates as the block $A_{j}$ (since both blocks have type 1 and
are strict, so their
number of indeterminates equals their number of full paradiagonals minus $1$).

Now, from each column of $A_{j}$, we subtract a scalar multiple of the
corresponding column of $A_{i}$\ \ \ \ \footnote{Here, we are counting the
columns of a block from the right. That is, the ``corresponding column'' of
the $3$-rd-from-the-right column of $A_{j}$ is the $3$-rd-from-the-right
column of $A_{i}$.
The reason why this ``corresponding column'' always exists is that the block
$A_{i}$ has at least as many columns as the block $A_{j}$.} (choosing the
scalar factor in such a way that the subtraction will turn the bottommost full
paradiagonal of $A_{j}$ into $\left(  0,0,\ldots,0\right)  $). As a result of
this subtraction, the indeterminates in $A_{j}$ are replaced by
``quasi-indeterminates'' (i.e., differences of the form ``indeterminate minus
a scalar'' or ``indeterminate minus a scalar multiple of another
indeterminate''). However, these ``quasi-indeterminates'' are still uniformly
independently distributed over $\Fq$ when we evaluate our
probability, and we can apply a change of variables to transform them back
into distinct indeterminates; as a result, the block $A_{j}$ once again
becomes a slant matrix, but now one of type 0 (since its bottommost full
paradiagonal is $\left(  0,0,\ldots,0\right)  $).

Thus we have transformed the block $A_j$ into a new block, which is a slant
matrix of type 0. All other blocks of $M$ remain as they were in $M$.
The full matrix obtained through these transformations is called
$M^{\prime}$.

Here is an example: Assume that
\[
M=\left(
\begin{array}
[c]{ccccc}%
b & 2 & 3 & z & 5\\
a & b & 2 & y & z\\
1 & a & b & x & y\\
& 1 & a & 1 & x\\
&  & 1 &  & 1
\end{array}
\right)  .
\]
This multislant matrix has just two blocks:
\[
A_{i}=\left(
\begin{array}
[c]{ccc}%
b & 2 & 3\\
a & b & 2\\
1 & a & b\\
& 1 & a\\
&  & 1
\end{array}
\right)  ,\ \ \ \ \ \ \ \ \ \ A_{j}=\left(
\begin{array}
[c]{cc}%
z & 5\\
y & z\\
x & y\\
1 & x\\
& 1
\end{array}
\right)  ,
\]
both being of type 1. Now, we do what we said we would do: From each column of
$A_{j}$, we subtract a scalar multiple of the corresponding column of $A_{i}$. In
this case, the necessary scalar factor is $1$ (since both $A_{i}$ and $A_{j}$
have bottom elements $1$), so we are just subtracting from each column of
$A_{j}$ the corresponding column of $A_{i}$. The resulting matrix is%
\[
\widetilde{M}=\left(
\begin{array}
[c]{ccccc}%
b & 2 & 3 & z-2 & 5-3\\
a & b & 2 & y-b & z-2\\
1 & a & b & x-a & y-b\\
& 1 & a & 0 & x-a\\
&  & 1 &  & 0
\end{array}
\right)  .
\]
Note that the entries $x-a$, $y-b$ and $z-2$ in this matrix are
``quasi-indeterminates'', whereas the $5-3$ is just a constant.
If we now perform a change of variables that substitutes $x$, $y$ and $z$ for
$x-a$, $y-b$ and $z-2$ in this matrix (we can do this without changing the
singular probability, since these are independent indeterminates), then we
obtain%
\[
\left(
\begin{array}
[c]{ccccc}%
b & 2 & 3 & z & 5-3\\
a & b & 2 & y & z\\
1 & a & b & x & y\\
& 1 & a & 0 & x\\
&  & 1 &  & 0
\end{array}
\right)  ;
\]
this is our matrix $M^{\prime}$. It is again a strict multislant matrix, and
it differs from $M$ only in that the block $A_{j}$ has been replaced by a
block of type 0. Since our column operations have left the determinant
unchanged (and our substitutions have left the singular probability
unchanged), we have $\operatorname*{SiPr}M=\operatorname*{SiPr}\left(
M^{\prime}\right)  $.

Note how we used that the block $A_{i}$ has at least as many columns as the
block $A_{j}$ (indeed, this ensured that each column of $A_{j}$ had a
corresponding column of $A_{i}$ to subtract from it), and also how we used
that the block $A_{i}$ has at most as many indeterminates as the block $A_{j}$
(indeed, this ensured that after the subtraction of columns, the
indeterminates from $A_{i}$ got subtracted only from indeterminates in $A_{j}%
$, rather than from the attic entries). Moreover, the strictness of $M$
ensures that we have not subtracted any indeterminates from any constants.

The example we have just analyzed was representative of the general case. If
$M$ has more blocks besides $A_{i}$ and $A_{j}$, then these extra blocks are
left untouched by the subtractions and do not interfere with the argument. If
the bottom elements of $A_{i}$ and $A_{j}$ are not $1$ but other nonzero
elements of $\Fq$, then we will have to subtract nontrivial scalar
multiples of columns of $A_{j}$ rather than subtracting these columns
directly, but the argument will not essentially change.
\end{proof}

\begin{lemma}
\label{lem.multislant.redX}Let $M$ be a multislant matrix that is square. Let
$A_{j}$ be a block of $M$ that has type X. For any $v\in\Fq$, we
let $M^{j\rightarrow v}$ be the multislant matrix obtained from $M$ by
replacing the block $A_{j}$ by $A_{j}^{\rightarrow v}$. Then,
\[
\operatorname*{SiPr}M=\dfrac{1}{q}\sum_{v\in\Fq}%
\operatorname*{SiPr}\left(  M^{j\rightarrow v}\right)  .
\]

\end{lemma}

\begin{proof}
Consider the bottom element of $A_{j}$; this is an indeterminate (since
$A_{j}$ has type X). When we substitute a random element of $\Fq$
for each of the indeterminates appearing in $M$, this indeterminate becomes an
element of $\Fq$. More precisely, for each $v\in\Fq$,
this indeterminate becomes $v$ with probability $\dfrac{1}{q}$. Thus, the
claim follows from the law of total probability.
\end{proof}

\begin{lemma}
\label{lem.multislant.red01}Let $M$ be a multislant matrix that is square and
has $k$ blocks and signature $\left(  0,k-1,1\right)  $.

Thus, all but one blocks of $M$ have type 0, whereas the remaining block has
type 1. Let us refer to the latter block as the ``strange block''. Let
$M^{\prime}$ be the matrix obtained from $M$ by removing the bottommost row of
$M$ and the rightmost column of the strange block. (Note that this will cause
the strange block to disappear entirely if it had only one column.) Then:

\begin{enumerate}

\item[\textbf{(a)}]
The matrix $M^{\prime}$ is again a multislant matrix of signature
$\left(  k-1,0,1\right)  $ or $\left(  k-1,0,0\right)  $ (depending on whether
the strange block had more than one column or not).

\item[\textbf{(b)}]
We have $\operatorname*{SiPr}M=\operatorname*{SiPr}\left(
M^{\prime}\right)  $.

\end{enumerate}

\end{lemma}

\begin{proof}
\textbf{(a)} When we pass from $M$ to $M^{\prime}$, each block of type 0
becomes a block of type X (since it loses a row and thus loses its bottommost
full paradiagonal, which consisted of zeroes\footnote{and it has at least two
full paradiagonals, since it has more rows than columns}).
The strange block either
remains a block of type 1 (if it had more than one column), or disappears
entirely (if it didn't). These account for all blocks of $M^{\prime}$. Thus,
the matrix $M^{\prime}$ is a multislant matrix of signature $\left(
k-1,0,1\right)  $ or $\left(  k-1,0,0\right)  $ (depending on whether the
strange block had more than one column or not). \medskip

\textbf{(b)} The bottom row of $M$ has only one nonzero entry, which is the
bottom element of the strange block. Since the strange block has type 1, this
bottom element must be some nonzero element $v\in\Fq$. Thus,
expanding the determinant of $M$ along the bottom row yields $\det M=\pm
v\det\left(  M^{\prime}\right)  $ (since $v$ is the only nonzero entry in the
bottom row of $M$). Since $v$ is nonzero, this yields that $\det M$ vanishes
exactly when $\det\left(  M^{\prime}\right)  $ vanishes. Hence,
$\operatorname*{SiPr}M=\operatorname*{SiPr}\left(  M^{\prime}\right)  $.
\end{proof}

\begin{proof}
[Proof of Theorem \ref{thm.multislant.prob}.]Induct on the size of the matrix
$M$. Inside that induction step, apply strong induction on $2i+\ell$. So let
us consider a multislant matrix $M$ that is square and nonempty and has
signature $\left(  i,j,\ell\right)  $. Let $k=i+j+\ell$ be its number of
blocks. We must prove the equality (\ref{eq.thm.multislant.prob.eq}).

Without loss of generality, we assume that the multislant matrix $M$ is
strict.\footnote{Indeed,
assume that (\ref{eq.thm.multislant.prob.eq}) is proved in the case when $M$
is strict. In other words, $\operatorname*{SiPr}M=1-\gamma_{k}\left(
1-\dfrac{0^{\ell}}{q^{i}}\right)  $ whenever $M$ is strict. Now, let $M$ be
arbitrary (not necessarily strict). We shall refer to the indeterminates that
appear as attic entries in the blocks of $M$ as \emph{attic variables}. If we
substitute a random element of $\Fq$ for each of the attic
variables, then the matrix $M$ becomes strict, and thus (by our assumption)
$\operatorname*{SiPr}M$ becomes $1-\gamma_{k}\left(  1-\dfrac{0^{\ell}}{q^{i}%
}\right)  $. Hence, $\operatorname*{SiPr}M$ also equals $1-\gamma_{k}\left(
1-\dfrac{0^{\ell}}{q^{i}}\right)  $ before these substitutions. In other
words, (\ref{eq.thm.multislant.prob.eq}) holds for arbitrary $M$. This shows
that \textquotedblleft$M$ is strict\textquotedblright\ can indeed be assumed
without loss of generality.}

We are in one of the following four cases:

\textit{Case 1:} The matrix $M$ has at least one block of type X.

\textit{Case 2:} The matrix $M$ has no block of type X, but has more than $1$
block of type 1.

\textit{Case 3:} The matrix $M$ has no block of type X, and has exactly one
block of type 1.

\textit{Case 4:} The matrix $M$ has no block of type X, and has no block of
type 1.

Let us first consider Case 1. In this case, the matrix $M$ has at least one
block of type X. Let $A_{j}$ be this block. Thus, Lemma
\ref{lem.multislant.redX} yields%
\begin{equation}
\operatorname*{SiPr}M=\dfrac{1}{q}\sum_{v\in\Fq}%
\operatorname*{SiPr}\left(  M^{j\rightarrow v}\right)  ,
\label{pf.thm.multislant.prob.c1.1}%
\end{equation}
where $M^{j\rightarrow v}$ is defined as in Lemma \ref{lem.multislant.redX}.
Now, for any nonzero $v\in\Fq$, the matrix $M^{j\rightarrow v}$ has
signature $\left(  i-1,j,\ell+1\right)  $ (since it is obtained from $M$ by
replacing the type-X block $A_{j}$ by a type-1 block) and still has $k$
blocks, and thus satisfies%
\[
\operatorname*{SiPr}\left(  M^{j\rightarrow v}\right)  =1-\gamma_{k}\left(
1-\dfrac{0^{\ell+1}}{q^{i-1}}\right)
\]
(by the induction hypothesis of our strong induction, since $2\left(
i-1\right)  +\left(  \ell+1\right)  <2i+\ell$). Since $0^{\ell+1}=0$, this
simplifies to%
\begin{equation}
\operatorname*{SiPr}\left(  M^{j\rightarrow v}\right)  =1-\gamma_{k}.
\label{pf.thm.multislant.prob.c1.2}%
\end{equation}
On the other hand, the matrix $M^{j\rightarrow0}$ has signature $\left(
i-1,j+1,\ell\right)  $ (since it is obtained from $M$ by replacing the type-X
block $A_{j}$ by a type-0 block) and still has $k$ blocks, and thus satisfies%
\begin{equation}
\operatorname*{SiPr}\left(  M^{j\rightarrow0}\right)  =1-\gamma_{k}\left(
1-\dfrac{0^{\ell}}{q^{i-1}}\right)  \label{pf.thm.multislant.prob.c1.3}%
\end{equation}
(by the induction hypothesis of our strong induction, since $2\left(
i-1\right)  +\ell<2i+\ell$). Now, (\ref{pf.thm.multislant.prob.c1.1}) becomes%
\begin{align*}
\operatorname*{SiPr}M  &  =\dfrac{1}{q}\sum_{v\in\Fq}%
\operatorname*{SiPr}\left(  M^{j\rightarrow v}\right)  =\dfrac{1}{q}\left(
\operatorname*{SiPr}\left(  M^{j\rightarrow0}\right)  +\sum_{\substack{v\in
\Fq;\\v\neq0}}\operatorname*{SiPr}\left(  M^{j\rightarrow
v}\right)  \right) \\
&  =\dfrac{1}{q}\left(  \left(  1-\gamma_{k}\left(  1-\dfrac{0^{\ell}}%
{q^{i-1}}\right)  \right)  +\sum_{\substack{v\in\Fq;\\v\neq
0}}\left(  1-\gamma_{k}\right)  \right) \\
&  \ \ \ \ \ \ \ \ \ \ \ \ \ \ \ \ \ \ \ \ \left(  \text{by
(\ref{pf.thm.multislant.prob.c1.3}) and (\ref{pf.thm.multislant.prob.c1.2}%
)}\right) \\
&  =\dfrac{1}{q}\left(  \left(  1-\gamma_{k}\left(  1-\dfrac{0^{\ell}}%
{q^{i-1}}\right)  \right)  +\left(  q-1\right)  \left(  1-\gamma_{k}\right)
\right) \\
&  \ \ \ \ \ \ \ \ \ \ \ \ \ \ \ \ \ \ \ \ \left(  \text{since the }\sum\text{
sum has precisely }q-1\text{ addends}\right) \\
&  =1-\gamma_{k}\left(  1-\dfrac{0^{\ell}}{q^{i}}\right)
\ \ \ \ \ \ \ \ \ \ \left(  \text{by a straightforward computation}\right)  .
\end{align*}
Thus, (\ref{eq.thm.multislant.prob.eq}) has been proven in Case 1.

Let us next consider Case 2. In this case, the matrix $M$ has no block of type
X, but has more than $1$ block of type 1. Thus, the signature $\left(
i,j,\ell\right)  $ of $M$ satisfies $\ell>1$, so that $0^{\ell}=0$ and
$0^{\ell-1}=0$.

We assumed that $M$ has more than $1$ block of type 1. Thus, $M$ has at least
two distinct blocks of type 1. Let these two blocks be $A_{i}$ and $A_{j}$,
labelled in such a way that the block $A_{i}$ has at least as many columns as
the block $A_{j}$. Thus, Lemma \ref{lem.multislant.red11} yields that there is
a strict multislant matrix $M^{\prime}$ with the following properties:

\begin{itemize}
\item The matrix $M^{\prime}$ is square and has the same size as $M$.

\item It satisfies $\operatorname*{SiPr}M=\operatorname*{SiPr}\left(
M^{\prime}\right)  $.

\item The matrix $M^{\prime}$ differs from $M$ only in the block $A_{j}$ being
replaced by a new block, which has type 0.
\end{itemize}

Consider this matrix $M^{\prime}$. By its third property, this matrix
$M^{\prime}$ has signature $\left(  i,j+1,\ell-1\right)  $, and thus still has
$k$ blocks. Hence,%
\[
\operatorname*{SiPr}\left(  M^{\prime}\right)  =1-\gamma_{k}\left(
1-\dfrac{0^{\ell-1}}{q^{i}}\right)
\]
(by the induction hypothesis of our strong induction, since $2i+\left(
\ell-1\right)  <2i+\ell$). In view of $\operatorname*{SiPr}%
M=\operatorname*{SiPr}\left(  M^{\prime}\right)  $ and $0^{\ell-1}=0=0^{\ell}%
$, this rewrites as
\[
\operatorname*{SiPr}M=1-\gamma_{k}\left(  1-\dfrac{0^{\ell}}{q^{i}}\right)  .
\]
Thus, (\ref{eq.thm.multislant.prob.eq}) has been proven in Case 2.

Let us next consider Case 3. In this case, the matrix $M$ has no block of type
X, and has exactly one block of type 1. Thus, the signature $\left(
i,j,\ell\right)  $ of $M$ satisfies $i=0$ and $\ell=1$. Hence, from
$k=\underbrace{i}_{=0}+j+\underbrace{\ell}_{=1}=j+1$, we obtain $j=k-1$ and
thus $\left(  \underbrace{i}_{=0},\underbrace{j}_{=k-1},\underbrace{\ell}%
_{=1}\right)  =\left(  0,k-1,1\right)  $. Thus, the matrix $M$ has signature
$\left(  0,k-1,1\right)  $. Hence, all but one blocks of $M$ have type 0,
whereas the remaining block has type 1. Let us refer to the latter block as
the ``strange block''. Let $M^{\prime}$ be the matrix obtained from $M$ by
removing the bottommost row of $M$ and the rightmost column of the strange
block. (Note that this will cause the strange block to disappear entirely if
it had only one column.) Then, Lemma \ref{lem.multislant.red01} \textbf{(a)}
yields that the matrix $M^{\prime}$ is again a multislant matrix of signature
$\left(  k-1,0,1\right)  $ or $\left(  k-1,0,0\right)  $ (depending on whether
the strange block had more than one column or not). Furthermore, Lemma
\ref{lem.multislant.red01} \textbf{(b)} yields that $\operatorname*{SiPr}%
M=\operatorname*{SiPr}\left(  M^{\prime}\right)  $.

Without loss of generality, we assume that $M$ is not a $1\times1$-matrix
(because if $M$ is a
$1\times1$-matrix, then it is easy to see that $k=1$ and $M=\left(
\begin{array}
[c]{c}%
1
\end{array}
\right)  $ and therefore $\operatorname*{SiPr}M=0=1-\gamma_{k}\left(
1-\dfrac{0^{\ell}}{q^{i}}\right)  $). Thus, if the strange block of $M$ has at
most one column, then we must have $k>1$ (because otherwise, the strange block
would be the only block of $M$, and therefore $M$ would be a $1\times
1$-matrix), and hence we have
\begin{equation}
\gamma_{k-1}\left(  1-\dfrac{1}{q^{k-1}}\right)  =\gamma_{k}
\label{pf.thm.multislant.prob.c3.gamma-rec}%
\end{equation}
in this case (by the definitions of $\gamma_{k-1}$ and $\gamma_{k}$). (Note
that the equality (\ref{pf.thm.multislant.prob.c3.gamma-rec}) would not hold
for $k=1$; this is why we had to handle the $1\times1$-matrix case separately.)

Now, we have already shown that $M^{\prime}$ is a multislant matrix of
signature $\left(  k-1,0,1\right)  $ or $\left(  k-1,0,0\right)  $ (depending
on whether the strange block had more than one column or not). Since this
matrix $M^{\prime}$ has smaller size than $M$, we can thus use the induction
hypothesis (of our first induction) to see that%
\begin{align*}
&  \operatorname*{SiPr}\left(  M^{\prime}\right) \\
&  =%
\begin{cases}
1-\gamma_{k}\left(  1-\dfrac{0^{1}}{q^{k-1}}\right)  , & \text{if the strange
block had more than one column};\\
1-\gamma_{k-1}\left(  1-\dfrac{0^{0}}{q^{k-1}}\right)  , & \text{otherwise}%
\end{cases}
\\
&  =%
\begin{cases}
1-\gamma_{k}, & \text{if the strange block had more than one column};\\
1-\gamma_{k-1}\left(  1-\dfrac{1}{q^{k-1}}\right)  , & \text{otherwise}%
\end{cases}
\\
&  =1-\gamma_{k}\ \ \ \ \ \ \ \ \ \ \left(  \text{by
(\ref{pf.thm.multislant.prob.c3.gamma-rec})}\right)  .
\end{align*}
Hence,%
\begin{align*}
\operatorname*{SiPr}M
&=\operatorname*{SiPr}\left(  M^{\prime}\right)
=1-\gamma_{k} \\
&=1-\gamma_{k}\left(  1-\dfrac{0^{\ell}}{q^{i}}\right)
\ \ \ \ \ \ \ \ \ \ \left(  \text{since }\ell=1\text{ and thus }0^{\ell
}=0\right)  .
\end{align*}
Thus, (\ref{eq.thm.multislant.prob.eq}) has been proven in Case 3.

Let us finally consider Case 4. In this case, the matrix $M$ has no block of
type X, and has no block of type 1. In other words, $i=0$ and $\ell=0$. All
blocks of $M$ have type 0 (since $M$ has no block of type X and no block of
type 1). Thus, the bottom row of $M$ is $\left(  0,0,\ldots,0\right)  $.
Consequently, we have $\det M=0$, so that $\operatorname*{SiPr}M=1$. Comparing
this with%
\begin{align*}
1-\gamma_{k}\left(  1-\dfrac{0^{\ell}}{q^{i}}\right)   &  =1-\gamma
_{k}\underbrace{\left(  1-\dfrac{0^{0}}{q^{0}}\right)  }_{=0}%
\ \ \ \ \ \ \ \ \ \ \left(  \text{since }\ell=0\text{ and }i=0\right) \\
&  =1,
\end{align*}
we find $\operatorname*{SiPr}M=1-\gamma_{k}\left(  1-\dfrac{0^{\ell}}{q^{i}%
}\right)  $. Thus, (\ref{eq.thm.multislant.prob.eq}) has been proven in Case 4.

Hence, we have proved (\ref{eq.thm.multislant.prob.eq}) in all four cases.
This completes the induction step, and thus Theorem \ref{thm.multislant.prob}
is proved.
\end{proof}

\subsection{\label{subsec.proof-shifted-staircase}Proof of Theorem \ref{thm.n.staircase} (ii)}

We shall now apply Theorem \ref{thm.multislant.prob} to Jacobi--Trudi matrices of $p$-shifted $n$-staircases.

\begin{lemma}
\label{multistair}
Let $p \leq n \leq k-1$.
Let $\lambda$ be the $p$-shifted $n$-staircase of length $k$.
Then, there is some matrix $J'$ obtained by permuting the columns of $J(\lambda)$ which is a multislant matrix
of signature $(p, n-p, 1)$.
(See Definition~\ref{def.multislant} for the definition of the signature.)
\end{lemma}

\begin{proof}
By definition,
$\lambda = (p+(k-1)n,\ p+(k-2)n,\ \ldots,\ p+n,\ p)$.
Thus, $\lambda_i = p + (k-i)n$ for each $i \in [k]$.

Define a $k\times k$-matrix $M = \tup{M_{i,j}}_{i,j\in[k]}$
by
\[
M_{i,j} = \lambda_{i}-i+j
= p + kn - i\tup{n+1} + j
\qquad \text{for all $i,j\in[k]$.}
\]
Thus, $J(\lambda) = \left( h_{M_{i,j}} \right)_{i,j \in [k]}$.

We observe that any two entries lying in the same column
of $M$ are congruent modulo $n+1$.
For each index $q \leq n+1$,
the columns $q,\ q+(n+1),\ q+2(n+1),\ \ldots$ of $M$
contain exactly the entries of $M$ that are congruent to $M_{q,q} = \lambda_q$ modulo $n+1$.
Thus, the submatrix of $J(\lambda)$ consisting of
the corresponding columns is a slant matrix\footnote{The
assumption $n \leq k-1$ ensures that this submatrix
is nonempty. The assumption $p \leq n$ ensures that all
of its basement entries are $0$. The definition of $M$
along with the fact that
$J(\lambda) = \left( h_{M_{i,j}} \right)_{i,j \in [k]}$
ensures that each paradiagonal is constant. The remaining
requirements in the definition of a slant matrix are easily
verified.}, and the slant matrices obtained for different
indices $q \leq n+1$ are disjoint.
The bottom elements of these slant matrices are the last $n+1$
entries of the last row of $J(\lambda)$; these are
$h_p, h_{p-1}, \ldots, h_1, 1, 0, 0, \ldots, 0$.

Permuting the columns of $J(\lambda)$ in such a way that
each of these slant matrices appears as a contiguous block,
we thus obtain a multislant matrix of signature
$(p, n-p, 1)$.
%
%
%
%
%
\end{proof}

\begin{proof}[Proof of Theorem \ref{thm.n.staircase} \textbf{(ii)}]
Lemma~\ref{multistair} shows that, up to permutation of columns,
$J(\lambda)$ is a multislant matrix of signature $(p, n-p, 1)$
(and thus with $n+1$ blocks).
Hence, Theorem \ref{thm.multislant.prob} yields
\[
\operatorname*{SiPr}\tup{J\tup{\lambda}}
= 1 - \gamma_{n+1} \underbrace{\tup{1 - \dfrac{0^1}{q^p}}}_{= 1}
= 1 - \gamma_{n+1}
= 1 - \prod_{i=1}^{n} \left(  1-\dfrac{1}{q^{i}}\right)
\]
(by definition of $\gamma_{n+1}$).
Since $P(s_{\lambda} \mapsto 0) = \operatorname*{SiPr}\tup{J\tup{\lambda}}$,
this is precisely the claim of Theorem~\ref{thm.n.staircase} \textbf{(ii)}.
\end{proof}

Theorem~\ref{thm.n.staircase} \textbf{(i)} can be proved similarly
(but, as mentioned above, also follows from elementary
spans-and-independence reasoning because of the distinctness of
all indeterminates in the matrix\footnote{See
\cite[Theorem 1]{Haglund} for a closely related result with a
very similar proof. In fact, if not for an entry of $J\tup{\lambda/\mu}$
being $h_0 = 1$, Theorem~\ref{thm.n.staircase} \textbf{(i)} would be
a particular case of \cite[Theorem 1]{Haglund}.}).

Having proved Theorem~\ref{thm.n.staircase}, we can see that for $n=1$
we have
$P(s_{\lambda} \mapsto 0) = \dfrac{1}{q}$, and for $n=2$ we have
$P(s_{\lambda} \mapsto 0) = \dfrac{q^2+q-1}{q^3}$
recovering Theorem 6.5 and proving Conjecture 10.1 from Anzis et al. \cite{Anzis18}.

\section{Conjugating the skew partition}

Next, we shall prove a general result that generalizes
\cite[Corollary 3.3]{Anzis18} from partitions to skew
partitions.
Recall that $\lambda^{t}$ denotes the conjugate of a
partition $\lambda$.

\begin{theorem}
\label{thm.transpose}Let $\lambda/\mu$ be a skew partition. Let $a\in
\Fq$. Then,
\[
P\left(  s_{\lambda/\mu}\mapsto a\right)  =P\left(  s_{\lambda^{t}/\mu^{t}%
}\mapsto a\right)  .
\]
\end{theorem}

\begin{vershort}
The following observation will aid in the proof.

\begin{proposition}
\label{prop.eval}
Let $f \in \Lambda_{\le N}  :=  \mathbb{Z}\left[  h_{1},h_{2},\ldots, h_N\right]$ and  $a \in \mathbb{F}_q$.
Then
\begin{equation}
\label{e prop eval}
P\left(  f\mapsto a\right)  =
\dfrac{\left(  \text{\# of ring homomorphisms  $\Lambda_{\le N} \to \mathbb{F}_q$
that send  $f$ to  $a$} \right)  }{q^N}.
\end{equation}
\end{proposition}
\begin{proof}
By the universal property of polynomial rings, we know that
for each $N$-tuple $\left(  z_{1},z_{2}%
,\ldots,z_{N}\right)  \in\Fq^{N}$, there exists a unique ring
homomorphism $\varphi \colon \Lambda_{\leq N}\rightarrow\Fq$ that
sends the indeterminates $h_{1},h_{2},\ldots,h_{N}$ to $z_{1},z_{2}
,\ldots,z_{N}$, respectively. Thus, the $N$-tuples in
$\Fq^N$ are in bijection with the ring homomorphisms
$\Lambda_{\le N} \to \mathbb{F}_q$.
Hence the right side of \eqref{e prop eval}
agrees with the definition of $P\left(  f\mapsto a\right)$.
\end{proof}

\begin{proof}
[Proof of Theorem \ref{thm.transpose}.] Let us identify $\mathcal{P}$ with
$\Lambda$ as in \S \ \ref{sec.def}. For each positive integer $n$, let
$e_{n}=s_{\left(  1^{n}\right)  }\in\Lambda=\mathcal{P}$ be the $n$-th
elementary symmetric function.

From the theory of symmetric functions (\cite[\S 7.6]{EC2} or \cite[(2.7)]%
{Macdonald}), it is known that there is an involutive\footnote{A map is said
to be \emph{involutive} if it is its own inverse.} ring automorphism
$\omega \colon \Lambda\rightarrow\Lambda$ (known as the \emph{omega involution} or as
the \emph{fundamental involution}) defined by setting
\[
\omega\left(  h_{n}\right)  =e_{n}\qquad\text{for all }n\geq1.
\]
Now, the elementary symmetric function $e_{n}$ can be written as a polynomial in $h_{1},h_{2},\ldots,h_{n}$,
which follows, for example, from the Jacobi--Trudi identity:
$e_{n}=s_{\left(  1^{n}\right)  }=\det\left(  h_{1+i-j}\right)  _{1\leq i\leq
n,\ 1\leq j\leq n}$. Applying  $\omega$ to this, we see that
$h_{n}$ can likewise be written as a polynomial in $e_{1},e_{2},\ldots,e_{n}$.
Thus
\begin{align*}
\Lambda_{\le N} \, = &\  \tup{\text{subring of $\Lambda$ generated by $h_1, \dots, h_N$} } \\
            = &\  \tup{\text{subring of $\Lambda$ generated by $e_1, \dots, e_N$} },
\end{align*}
and $\omega$ restricts to an involutive ring automorphism
 $\omega_N \colon \Lambda_{\le N} \to \Lambda_{\le N}$.

Now, given a skew Schur function  $s_{\lambda/\mu}$, choose  $N$ such that
$s_{\lambda/\mu} \in \Lambda_{\le N}$.  Then,
$\omega_N\left(  s_{\lambda/\mu}\right) = \omega\left(  s_{\lambda/\mu}\right)
=s_{\lambda^{t} /\mu^{t}}$ (see, e.g., \cite[Theorem 7.15.6]{EC2} or \cite[(5.6)]{Macdonald}).
Thus, there is a map
\begin{align*}
& \left(  \text{ring homomorphisms  $\Lambda_{\le N} \to \mathbb{F}_q$
that send  $s_{\lambda^{t} /\mu^{t}}$ to  $a$} \right)  \\
\to & \left(  \text{ring homomorphisms  $\Lambda_{\le N} \to \mathbb{F}_q$
that send  $s_{\lambda /\mu}$ to  $a$} \right)
\end{align*}
that sends each
$\varphi$ to $\varphi \circ \omega_N$.
This map is furthermore a bijection, since $\omega_N$ is an automorphism.
Hence, the number of ring homomorphisms  $\Lambda_{\le N} \to \mathbb{F}_q$
that send  $s_{\lambda /\mu}$ to  $a$ does not change when we replace
$\lambda / \mu$ by $\lambda^t / \mu^t$.
By Proposition \ref{prop.eval}, this entails that
$P\left(  s_{\lambda/\mu}\mapsto a\right)  =P\left(  s_{\lambda^{t}/\mu^{t}%
}\mapsto a\right) $.
\end{proof}
\end{vershort}

\begin{verlong}
We will derive this from a more general result. First, we introduce a notation:

\begin{definition}
A ring endomorphism $\beta \colon \mathcal{P}\rightarrow\mathcal{P}$ will be called
\emph{friendly} if for each positive integer $n$, the image $\beta\left(
h_{n}\right)  $ can be written as a polynomial in $h_{1},h_{2},\ldots,h_{n}$.
\end{definition}

\begin{proposition}
\label{prop.friendly-aut}Let $\beta \colon \mathcal{P}\rightarrow\mathcal{P}$ be a
ring automorphism such that both $\beta$ and $\beta^{-1}$ are friendly. Let
$f\in\mathcal{P}$ and $a\in\Fq$. Then,%
\[
P\left(  f\mapsto a\right)  =P\left(  \beta\left(  f\right)  \mapsto a\right)
.
\]

\end{proposition}

\begin{proof}
[Proof of Proposition \ref{prop.friendly-aut}.] Choose $N\in\mathbb{N}$ such
that $f$ only involves the indeterminates $h_{1},h_{2},\ldots,h_{N}$. Thus, we
can write $f$ as $f\left(  h_{1},h_{2},\ldots,h_{N}\right)  $.

Let $\Lambda_{\leq N}$ be the subring of $\Lambda$ generated by $h_{1}%
,h_{2},\ldots,h_{N}$. Hence, our choice of $N$ guarantees that $f\in
\Lambda_{\leq N}$. The definition of $P\left(  f\mapsto a\right)  $ yields%
\begin{equation}
P\left(  f\mapsto a\right)  = \dfrac{\left(  \text{\# of }\left(  z_{1}%
,z_{2},\ldots,z_{N}\right)  \in\Fq^{N}\mid f\left(  z_{1}%
,z_{2},\ldots,z_{N}\right)  =a\right)  }{\left(  \text{\# of all }\left(
z_{1},z_{2},\ldots,z_{N}\right)  \in\Fq^{N}\right)  }%
.\label{pf.prop.friendly-aut.P1}%
\end{equation}

However, $\mathcal{P}_{\leq N}$ is the polynomial ring over $\mathbb{Z}$ in
the indeterminates $h_{1},h_{2},\ldots,h_{N}$ (since $h_{1},h_{2},\ldots
,h_{N}$ are algebraically independent by definition). Hence, by the universal
property of polynomial rings, for each $N$-tuple $\left(  z_{1},z_{2}%
,\ldots,z_{N}\right)  \in\Fq^{N}$, there exists a unique ring
homomorphism $\varphi \colon \mathcal{P}_{\leq N}\rightarrow\Fq$ that
sends these indeterminates $h_{1},h_{2},\ldots,h_{N}$ to $z_{1},z_{2}%
,\ldots,z_{N}$, respectively. In other words, the map%
\begin{align*}
\left\{  \text{ring homomorphisms }\varphi \colon \mathcal{P}_{\leq N}\rightarrow
\Fq\right\}    & \rightarrow\Fq^{N},\\
\varphi & \mapsto\left(  \varphi\left(  h_{1}\right)  ,\varphi\left(
h_{2}\right)  ,\ldots,\varphi\left(  h_{N}\right)  \right)
\end{align*}
is a bijection. It restricts to a bijection
\begin{align*}
& \text{from }\left\{  \text{ring homomorphisms }\varphi \colon \mathcal{P}_{\leq
N}\rightarrow\Fq\ \mid\ \varphi\left(  f\right)  =a\right\}  \\
& \text{to }\left\{  \left(  z_{1},z_{2},\ldots,z_{N}\right)  \in
\Fq^{N}\ \mid\ f\left(  z_{1},z_{2},\ldots,z_{N}\right)
=a\right\}
\end{align*}
(because for any ring homomorphism $\varphi \colon \mathcal{P}_{\leq N}%
\rightarrow\Fq$, we have the equality
$\varphi\left(  f\right)  =f\left(
\varphi\left(  h_{1}\right)  ,\varphi\left(  h_{2}\right)  ,\ldots
,\varphi\left(  h_{N}\right)  \right)  $). Hence, by the bijection principle,%
\begin{align}
& \left(  \text{\# of ring homomorphisms }\varphi \colon \mathcal{P}_{\leq
N}\rightarrow\Fq\ \mid\ \varphi\left(  f\right)  =a\right)
\nonumber\\
& =\left(  \text{\# of }\left(  z_{1},z_{2},\ldots,z_{N}\right)  \in
\Fq^{N}\mid f\left(  z_{1},z_{2},\ldots,z_{N}\right)  =a\right)
.\label{pf.prop.friendly-aut.1}%
\end{align}

Since $\beta$ is friendly, it is easily seen that $\beta\left(  \Lambda_{\leq
N}\right)  \subseteq\Lambda_{\leq N}$. Therefore, we can restrict $\beta$ to
the subring $\Lambda_{\leq N}$ of $\Lambda$, thus obtaining a ring
endomorphism $\beta_{N} \colon \Lambda_{\leq N}\rightarrow\Lambda_{\leq N}$.
Likewise, we can obtain an endomorphism $\left(  \beta^{-1}\right)
_{N} \colon \Lambda_{\leq N}\rightarrow\Lambda_{\leq N}$ by restricting $\beta^{-1}$
to $\Lambda_{\leq N}$ (since $\beta^{-1}$ is friendly). These two restricted
endomorphisms $\beta_{N}$ and $\left(  \beta^{-1}\right)  _{N}$ are mutually
inverse, and thus are automorphisms.

Now, let $g=\beta\left(  f\right)  $. From $f\in\Lambda_{\leq N}$, we obtain
$\beta\left(  f\right)  \in\beta\left(  \Lambda_{\leq N}\right)
\subseteq\Lambda_{\leq N}$, so that $g=\beta\left(  f\right)  \in\Lambda_{\leq
N}$. In other words, the polynomial $g$ only involves the indeterminates
$h_{1},h_{2},\ldots,h_{N}$. Therefore, the same argument that gave us
(\ref{pf.prop.friendly-aut.1}) can be applied to $g$ instead of $f$, and this
yields%
\begin{align}
& \left(  \text{\# of ring homomorphisms }\varphi \colon \mathcal{P}_{\leq
N}\rightarrow\Fq\ \mid\ \varphi\left(  g\right)  =a\right)
\nonumber\\
& =\left(  \text{\# of }\left(  z_{1},z_{2},\ldots,z_{N}\right)  \in
\Fq^{N}\mid g\left(  z_{1},z_{2},\ldots,z_{N}\right)  =a\right)
.\label{pf.prop.friendly-aut.2}%
\end{align}
Likewise, the same argument that gave us (\ref{pf.prop.friendly-aut.P1})
yields%
\begin{equation}
P\left(  g\mapsto a\right) = \dfrac{\left(  \text{\# of }\left(  z_{1}%
,z_{2},\ldots,z_{N}\right)  \in\Fq^{N}\mid g\left(  z_{1}%
,z_{2},\ldots,z_{N}\right)  =a\right)  }{\left(  \text{\# of all }\left(
z_{1},z_{2},\ldots,z_{N}\right)  \in\Fq^{N}\right)  }%
.\label{pf.prop.friendly-aut.P2}%
\end{equation}

However, $\beta_{N}$ is a ring automorphism of $\Lambda_{\leq N}$, and
satisfies $g=\beta\left(  f\right)  =\beta_{N}\left(  f\right)  $ (since
$\beta_{N}$ is a restriction of $\beta$). Thus, there is a bijection%
\begin{align*}
& \left\{  \text{ring homomorphisms }\varphi \colon \mathcal{P}_{\leq N}%
\rightarrow\Fq\ \mid\ \varphi\left(  g\right)  =a\right\}  \\
& \mapsto\left\{  \text{ring homomorphisms }\varphi \colon \mathcal{P}_{\leq
N}\rightarrow\Fq\ \mid\ \varphi\left(  f\right)  =a\right\}
\end{align*}
given by $\varphi\mapsto\varphi\circ\beta_{N}$ (its inverse sends $\varphi$ to
$\varphi\circ\left(  \beta_{N}\right)  ^{-1}$). Hence, by the bijection
principle, we have%
\begin{align*}
& \left(  \text{\# of ring homomorphisms }\varphi \colon \mathcal{P}_{\leq
N}\rightarrow\Fq\ \mid\ \varphi\left(  g\right)  =a\right)  \\
& =\left(  \text{\# of ring homomorphisms }\varphi \colon \mathcal{P}_{\leq
N}\rightarrow\Fq\ \mid\ \varphi\left(  f\right)  =a\right)  .
\end{align*}
In other words, the left hand sides of the equalities
(\ref{pf.prop.friendly-aut.2}) and (\ref{pf.prop.friendly-aut.1}) are equal.
Therefore, so are their right hand sides. In other words,%
\begin{align*}
& \left(  \text{\# of }\left(  z_{1},z_{2},\ldots,z_{N}\right)  \in
\Fq^{N}\mid g\left(  z_{1},z_{2},\ldots,z_{N}\right)  =a\right)
\\
& =\left(  \text{\# of }\left(  z_{1},z_{2},\ldots,z_{N}\right)  \in
\Fq^{N}\mid f\left(  z_{1},z_{2},\ldots,z_{N}\right)  =a\right)  .
\end{align*}
Thus, the right hand sides of the equalities (\ref{pf.prop.friendly-aut.P2})
and (\ref{pf.prop.friendly-aut.P1}) are equal. Hence, so are their left hand
sides. In other words,%
\[
P\left(  g\mapsto a\right)  =P\left(  f\mapsto a\right)  .
\]
In view of $g=\beta\left(  f\right)  $, this rewrites as $P\left(
\beta\left(  f\right)  \mapsto a\right)  =P\left(  f\mapsto a\right)  $. Thus,
Proposition \ref{prop.friendly-aut}.

\end{proof}

\begin{proof}
[Proof of Theorem \ref{thm.transpose}.] Let us identify $\mathcal{P}$ with
$\Lambda$ as in \S \ \ref{sec.def}. For each positive integer $n$, let
$e_{n}=s_{\left(  1^{n}\right)  }\in\Lambda=\mathcal{P}$ be the $n$-th
elementary symmetric function.

From the theory of symmetric functions (\cite[\S 7.6]{EC2} or \cite[(2.7)]%
{Macdonald}), it is known that there is an involutive\footnote{A map is said
to be \emph{involutive} if it is its own inverse.} ring automorphism
$\omega \colon \Lambda\rightarrow\Lambda$ (known as the \emph{omega involution} or as
the \emph{fundamental involution}) defined by setting
\[
\omega\left(  h_{n}\right)  =e_{n}\qquad\text{for all }n\geq1.
\]
Since $e_{n}$ can be written as a polynomial in $h_{1},h_{2},\ldots,h_{n}$
(for example, this follows from the Jacobi--Trudi identity, which yields
$e_{n}=s_{\left(  1^{n}\right)  }=\det\left(  h_{1+i-j}\right)  _{1\leq i\leq
n,\ 1\leq j\leq n}$), we thus conclude that this automorphism $\omega$ is
friendly. Its inverse $\omega^{-1}=\omega$ is friendly as well (since $\omega$
is involutive).

Hence, Proposition \ref{prop.friendly-aut} (applied to $\beta=\omega$ and
$f=s_{\lambda/\mu}$) yields%
\begin{equation}
P\left(  s_{\lambda/\mu}\mapsto a\right)  =P\left(  \omega\left(
s_{\lambda/\mu}\right)  \mapsto a\right)  .\label{pf.thm.transpose.1}%
\end{equation}

It is furthermore well-known (\cite[Theorem 7.15.6]{EC2} or \cite[(5.6)]%
{Macdonald}) that $\omega\left(  s_{\lambda/\mu}\right)  =s_{\lambda^{t}%
/\mu^{t}}$. In light of this, we can rewrite (\ref{pf.thm.transpose.1}) as
$P\left(  s_{\lambda/\mu}\mapsto a\right)  =P\left(  s_{\lambda^{t}/\mu^{t}%
}\mapsto a\right)  $. This proves Theorem \ref{thm.transpose}.
\end{proof}
\end{verlong}

\section{Block Staircases}

Using Theorem~\ref{thm.transpose}, we can extend our results on
$p$-shifted $n$-staircases to their conjugates. These conjugates can
be described independently.
We will use the \emph{exponential notation} for partitions, i.e.,
we will write $\lambda =((c_1)^{a_1}, (c_2)^{a_2}, \ldots, (c_k)^{a_k})$
as a shorthand for $\lambda = ( \underbrace{c_1,c_1,\ldots,c_1}_{a_1
\text{ times}},  \underbrace{c_2,c_2,\ldots,c_2}_{a_2
\text{ times}}, \ldots, \underbrace{c_k,c_k,\ldots,c_k}_{a_k
\text{ times}})$.

\begin{definition}
Let $p > 0$, $n > 0$ and $k \geq 0$ be integers.
We define the corresponding \emph{block staircase} to be the partition
\[
\lambda = ( k^p, (k-1)^n, (k-2)^n, \ldots,  2^n, 1^n
)
\]
(with length $\ell(\lambda) = p+(k-1)n$).
\end{definition}

It is easy to see that a block staircase is the conjugate of a $p$-shifted $n$-staircase partition.
Hence, we can apply Theorem~\ref{thm.transpose} to Theorem~\ref{thm.n.staircase} and Conjecture~\ref{bigoutwardstaircases}
and obtain the following corollary and conjecture:
\begin{corollary}
\label{crl.n.staircase}

Let $\lambda$ be a block staircase.

\begin{enumerate}

\item[\textbf{(i)}]
Let $k < n+1$ (with $p$ arbitrary). Then,
\[
P(s_{\lambda} \mapsto 0) =
1 -
\begin{cases}
\prod\limits_{i=1}^{k-1} \left(  1-\dfrac{1}{q^{i}}\right) ,
& \text{ if } p \leq k-1; \\
\prod\limits_{i=1}^{k} \left(  1-\dfrac{1}{q^{i}}\right) ,
& \text{ if } p > k-1.
\end{cases}
\]

\item[\textbf{(ii)}]
Let $p \leq n$ and $k \geq n+1$. Then,
\[
P(s_{\lambda} \mapsto 0) = 1- \prod_{i=1}^{n} \left(  1-\dfrac{1}{q^{i}}\right) .
\]

\end{enumerate}

\end{corollary}

\begin{conjecture}
\label{bigoutwardblockstaircases}
Let $\lambda$ be a block staircase with $k \geq n+1$ with $p > n$.  Then,
\[
P(s_{\lambda} \mapsto 0) = 1-\prod_{i=1}^{n+1} \left(  1-\dfrac{1}{q^{i}}\right).
\]
\end{conjecture}

\section{The ribbon equidistribution}

In this section, we move our attention to a particular class of skew partitions called ribbons. We begin with their definition (see, e.g., {\cite[\S7.17]{EC2}}):

\begin{definition}
A skew partition $\lambda/\mu$ is \emph{connected} if the interior of the Young diagram of $\lambda/\mu$ (regarded as the union of its boxes) is connected.

A \emph{ribbon} is a connected skew partition $\lambda/\mu$ whose Young diagram does not contain any $2 \times 2$ block of boxes.
\end{definition}



\begin{example}
\label{ribbontableaux}
Of the three skew partitions below, only the first is a ribbon, since \textbf{(ii)} contains a $2\times 2$ block (shaded), and \textbf{(iii)} is disconnected.

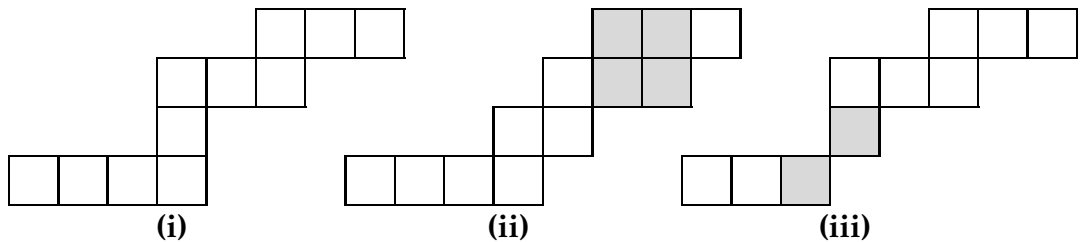
\begin{minipage}[t]{0.30\textwidth}
\begin{center}
\ydiagram{5+3,3+3,3+1,4}

\textbf{(i)}\vphantom{$\int^a$}
\end{center}
\end{minipage}
\begin{minipage}[t]{0.30\textwidth}
\begin{center}
\ydiagram[*(gray!30) ]{5+2,5+2}
*[*(white)]{5+3,4+3,3+2,4}

\textbf{(ii)}\vphantom{$\int^a$}
\end{center}
\end{minipage}
\begin{minipage}[t]{0.30\textwidth}
\begin{center}
\ydiagram[*(gray!30) ]{0,0,3+1,2+1}
*[*(white)]{5+3,3+3,3+1,3}

\textbf{(iii)}\vphantom{$\int^a$}
\end{center}
\end{minipage}


\end{example}

Let us revisit Jacobi--Trudi matrices. For the ribbon  $\lambda/\mu= (8,6,4,4)/(5,3,3)$ in Example \ref{ribbontableaux} (i), we can write down $J(\lambda/\mu)$ as follows:
\begin{align*}
J(\lambda/\mu)
=
\begin{pmatrix}
h_3 & h_6 & h_7 & h_{11} \\
1 & h_3 & h_4 & h_8 \\
0 & 1 & h_1 & h_5 \\
0 & 0 & 1 & h_4
\end{pmatrix} .
\end{align*}

\begin{lemma}
\label{lemma.Hessenberg}
Let $\lambda= (\lambda_1, \lambda_2, \ldots, \lambda_{\ell})$  and $\mu= (\mu_1, \mu_2, \ldots, \mu_{\ell}) \subseteq \lambda$ be two partitions such that $\lambda/\mu$ is a ribbon.
Suppose that $\lambda_1> \mu_1$ and $\lambda_\ell > \mu_\ell = 0$.\\

The Jacobi--Trudi matrix $J(\lambda/\mu)$ of the ribbon $\lambda/\mu$ has the following properties
(where the notation $\tup{J(\lambda/\mu)}_{i,j}$ means the $\tup{i,j}$-th entry of this matrix):

\begin{enumerate}

\item[\textbf{(i)}]
We have $J(\lambda/\mu)_{j+1,j}=1$ for each $j \in [\ell-1]$.

\item[\textbf{(ii)}]
We have $J(\lambda/\mu)_{i,j}=0$ whenever
$i, j \in [\ell]$ satisfy $i > j+1$.

\item[\textbf{(iii)}]
Let $N= \lambda_\ell -\mu_l-1+\ell$. The upper right entry of $J(\lambda/\mu)$ is $h_N$ and the remaining entries of $J(\lambda/\mu)$ lie in $\mathbb{Z}\left[  h_{1},h_{2},h_{3},\ldots, h_{N-1}\right]$.

\end{enumerate}

\end{lemma}

\begin{remark}
Lemma \ref{lemma.Hessenberg} essentially claims that the matrix $J(\lambda / \mu)$ has the following structure (shown here for $\ell = 5$):
\[
J(\lambda / \mu) = \left(  h_{\lambda_{i}-\mu_{j}-i+j}\right)  _{i,j\in\left[  \ell\right]
}=\left(
\begin{array}
[c]{ccccc}%
\ast & \ast & \ast & \ast & h_{N}\\
1 & \ast & \ast & \ast & \ast\\
0 & 1 & \ast & \ast & \ast\\
0 & 0 & 1 & \ast & \ast\\
0 & 0 & 0 & 1 & \ast
\end{array}
\right),
\]
where each of the asterisks is an element of $\mathbb{Z}\left[  h_{1},h_{2},h_{3},\ldots, h_{N-1}\right]$.
\end{remark}

\begin{proof}[Proof of Lemma \ref{lemma.Hessenberg}.]
From $\lambda_\ell > \mu_\ell = 0$, we obtain $\ell=\ell\left(\lambda\right)$.

\textbf{(i)} Fix $j \in [\ell-1]$. A box of the skew diagram $Y(\lambda / \mu)$ will be called \emph{high} if it lies in one of rows $1, 2, \ldots, j$, and will be called \emph{low} if it lies in one of rows $j+1, j+2, \ldots, \ell$. Note that the diagram $\lambda / \mu$ contains both high and low boxes (since $\lambda_1> \mu_1$ and $\lambda_\ell > \mu_\ell = 0$). Since it is connected, it must thus contain a high box and a low box that are adjacent to one another. These two boxes thus have the form $(j, p)$ and $(j+1, p)$ for some $p > 0$. This $p$ then satisfies $\mu_j < p$ (since $(j, p) \in Y(\lambda / \mu)$) and $p \leq \lambda_{j+1}$ (since $(j+1, p)\in Y(\lambda / \mu)$). Thus, $\mu_j < p \leq \lambda_{j+1}$.

On the other hand, if we had $\lambda_{j+1}>\mu_{j}+1$, then the diagram $Y(\lambda / \mu)$ would contain a $2 \times 2$ block of boxes, namely the four boxes $(j, \mu_j + 1)$, $(j, \mu_j + 2)$, $(j+1, \mu_j + 1)$ and $(j+1, \mu_j + 2)$ (because $\lambda_j \geq \lambda_{j+1} \geq \mu_j + 2 \geq \mu_j + 1 > \mu_j \geq \mu_{j+1}$). But this would contradict the fact that $\lambda / \mu$ is a ribbon and thus has no such blocks.

Hence, we cannot have $\lambda_{j+1}>\mu_{j}+1$. Therefore, $\lambda_{j+1}\leq \mu_{j}+1$. Combined with $\mu_j < \lambda_{j+1}$, this yields $\lambda_{j+1}=\mu_{j}+1$, as desired.\\

\textbf{(ii)} Statement \textbf{(ii)} is equivalent to showing $\lambda_{i}-\mu_{j}-i+j < 0$ when $i > j+1$. This follows from \textbf{(i)}, since the indices $\lambda_{i}-\mu_{j}-i+j$ in the Jacobi--Trudi matrix $J(\lambda / \mu)$  are strictly increasing from left to right across each row.\\

\textbf{(iii)} Statement \textbf{(iii)} also follows likewise since these indices strictly increase across rows and strictly decrease down columns (making the upper right entry's index uniquely the largest one).
\end{proof}


\begin{theorem}
\label{thm.ribbon}
Let $\lambda/\mu$ be a ribbon, and recall that $s_{\lambda/\mu}$ is the corresponding Schur function.
Let $a\in \Fq$.
Then,
\[
P(s_{\lambda/\mu} \mapsto a) = 1/q.
\]
\end{theorem}

\begin{proof}
By Remark \ref{flushremark}, the skew Schur function $s_{\lambda / \mu}$ is unchanged when the diagram of $\lambda / \mu$ is translated. We can therefore assume that the diagram of $\lambda / \mu$ has been translated as far as possible to the northwest -- i.e., that we have $\lambda_1> \mu_1$ and $\lambda_\ell > \mu_\ell =0$, where $\ell$ is the length of $\lambda$.\\

Consider the submatrix of $J(\lambda/\mu)$ obtained by removing the first row and the last column of the matrix.
By Lemma \ref{lemma.Hessenberg} \textbf{(i)} and \textbf{(ii)}, this submatrix is upper-triangular with diagonal $1, 1, \ldots, 1$, so that its determinant is $1$.
Hence, the cofactor expansion of the determinant of $J(\lambda/\mu)$ along the first row has the form
\[
 s_{\lambda / \mu } = \det(J(\lambda/\mu))= (-1)^{\ell+1} h_N + g( h_1, h_2, \ldots,h_{N-1})
\]
for some polynomial $g( h_1, h_2, \ldots,h_{N-1}) \in \mathbb{Z}\left[  h_{1},h_{2},h_{3},\ldots, h_{N-1}\right]$,
by Lemma \ref{lemma.Hessenberg} \textbf{(iii)}.

Thus, an $N$-tuple $\tup{z_1, z_2, \ldots, z_N} \in \Fq^N$
satisfies
$s_{\lambda / \mu } \tup{z_1, z_2, \ldots, z_{N}} = a$ if
and only if
$(-1)^{\ell+1} z_N + g( z_1, z_2, \ldots,z_{N-1}) = a$.
Clearly, the latter equality has a unique solution for $z_N$
if $z_1, z_2, \ldots, z_{N-1}$ are given.

Hence, for any $\tup{z_1, z_2, \ldots, z_{N-1}} \in \Fq^{N-1}$, there is
exactly one value of $z_N \in \Fq$ that satisfies
$s_{\lambda / \mu } \tup{z_1, z_2, \ldots, z_{N}} = a$.
In other words, we have
\[
\left(  \text{\# of }\left(  z_{1},z_{2},\ldots,z_{N}\right)  \in\Fq^{N}
\ \mid\ (-1)^{n+1} z_N + g( z_1, z_2, \ldots,z_{N-1}) =a\right)
= q^{N-1}.
\]

From Definition \ref{def.mapstozero}, we have
\begin{align*}
P\left(  s_{\lambda / \mu }\mapsto a\right)
&=\dfrac{\left(  \text{\# of }\left(  z_{1}%
,z_{2},\ldots,z_{N}\right)  \in\Fq^{N}\text{ such that } s_{\lambda / \mu } \left(
z_{1},z_{2},\ldots,z_{N}\right)  =a\right)  }{\left(  \text{\# of all }\left(
z_{1},z_{2},\ldots,z_{N}\right)  \in\Fq^{N}\right)  }\\
\\
&=\dfrac{q^{N-1}}{q^N} =\dfrac{1}{q} .
\end{align*}
\end{proof}

\printbibliography

@article{Anzis18,
  title={Jacobi-{T}rudi Determinants over Finite Fields},
  author={Anzis, Ben and Chen, Shuli and Gao, Yibo and Kim, Jesse and Li, Zhaoqi and Patrias, Rebecca},
  journal={Annals of Combinatorics},
  volume={22},
  number={1},
  pages={447--489},
  year={2018},
  publisher={Springer}
}

@book{EC2,
  title={Enumerative Combinatorics: Volume 2},
  author={Stanley, Richard P.},
  year={2001},
  publisher={Cambridge University Press}
}

@book{Macdonald,
  title={Symmetric Functions and Hall Polynomials},
  author={Macdonald, Ian G.},
  year={1998},
  edition={2},
  publisher={Oxford Science Publications}
}

@ARTICLE{dwivedi2021rank,
  title     = "On the rank of Hankel matrices over finite fields",
  author    = "Dwivedi, Omesh Dhar and Grinberg, Darij",
  journal   = "Linear Algebra Appl.",
  publisher = "Elsevier BV",
  volume    =  641,
  pages     = "156--181",
  month     =  may,
  year      =  2022
}

@article{RPStan,
  doi = {10.1007/bf01608530},
  url = {https://doi.org/10.1007/bf01608530},
  year = {1998},
  month = dec,
  publisher = {Springer Science and Business Media {LLC}},
  volume = {2},
  number = {4},
  pages = {351--363},
  author = {Richard P. Stanley},
  title = {Spanning trees and a conjecture of {K}ontsevich},
  journal = {Annals of Combinatorics},
      eprint={math/9806055},
      archivePrefix={arXiv},
      primaryClass={math.CO},
}

@article{Weil,
  title={Numbers of solutions of equations in finite fields},
  author={Andre Weil},
  journal={Bulletin of the American Mathematical Society},
  year={1949},
  volume={55},
  issue={5},
  pages={497--508},
  doi={10.1090/S0002-9904-1949-09219-4}
}

@misc{Kontsevich,
  title={Gelfand Seminar talk, {R}utgers Univ.},
  author={Maxim Kontsevich},
  year={1997}
}

@article{Stembridge,
  title={Counting Points on Varieties over Finite Fields Related to a Conjecture of {K}ontsevich},
  author={John R. Stembridge},
  journal={Annals of Combinatorics},
  volume={2},
  year={1998},
  pages={365--385},
  doi={10.1007/BF01608531}
}

@article{Elkies,
  author = {Noam D. Elkies},
  title = {On finite sequences satisfying linear recursions},
  journal={New York Journal of Mathematics},
  volume = {8},
  year ={2002}
}

@article{Haglund,
  doi = {10.1006/aama.1998.0582},
  url = {https://doi.org/10.1006/aama.1998.0582},
  year = {1998},
  month = may,
  volume = {20},
  number = {4},
  pages = {450--487},
  author = {James Haglund},
  title = {$q$-Rook Polynomials and Matrices over Finite Fields},
  journal = {Advances in Applied Mathematics}
}

@article{Belkale2003,
  doi = {10.1215/s0012-7094-03-11615-4},
  url = {https://doi.org/10.1215/s0012-7094-03-11615-4},
  year = {2003},
  month = jan,
  publisher = {Duke University Press},
  volume = {116},
  number = {1},
  author = {Prakash Belkale and Patrick Brosnan},
  title = {Matroids, motives, and a conjecture of {K}ontsevich},
  journal = {Duke Mathematical Journal}
}

@book{Johnson2020,
  title={An introduction to $q$-analysis},
  author={Johnson, Warren P.},
  year={2020},
  publisher={American Mathematical Society}
}
\end{document}